\documentclass[a4paper,11pt]{article}
\usepackage[utf8]{inputenc}
\usepackage[T1]{fontenc}
\usepackage[english]{babel}
\usepackage[left=2cm,right=2cm,top=2cm,bottom=2cm]{geometry}
\usepackage{caption}
\usepackage{subcaption}
\usepackage[titletoc,title]{appendix}
\usepackage{amsthm}
\usepackage{csquotes}
\usepackage{mathtools, bm}
\usepackage{amsmath, amsfonts,amssymb,thmtools}
\usepackage{yhmath}
\usepackage{dsfont}
\usepackage{stmaryrd}
\usepackage{enumitem} 
\usepackage{mdframed}
\usepackage{mathrsfs} 
\usepackage{graphicx}
\usepackage{float}
\usepackage[useregional]{datetime2}
\usepackage{titlesec}
\usepackage[hidelinks]{hyperref}
\usepackage{tikz}

\definecolor{darkpastelgreen}{rgb}{0.01, 0.75, 0.24}
\hypersetup{colorlinks = true, linkcolor = blue, citecolor = darkpastelgreen}

\newcommand{\R}{\ensuremath{\mathbf{R}}}

\newcommand{\F}{\ensuremath{\mathbf{F}}}
\newcommand{\Q}{\ensuremath{\mathbf{Q}}}
\newcommand{\C}{\ensuremath{\mathbf{C}}}
\newcommand{\Z}{\ensuremath{\mathbf{Z}}}

\newcommand{\K}{\ensuremath{\mathrm{K}}}
\newcommand{\SU}{\ensuremath{\mathrm{SU}}}
\newcommand{\Sim}{\ensuremath{\mathrm{S}}}
\newcommand{\B}{\ensuremath{\mathrm{B}}}
\newcommand{\QQ}{\ensuremath{\mathrm Q}}

\newcommand{\dd}{\ensuremath{\mathrm{d}}}
\renewcommand{\H}{\ensuremath{\mathbb{H}}}

\newcommand{\T}{\ensuremath{\mathbb{T}}}
\newcommand{\xx}{\ensuremath{\mathbf{x}}}

\newcommand{\yy}{\ensuremath{\mathbf{y}}}
\newcommand{\mm}{\ensuremath{\mathbf{m}}}

\renewcommand{\mod}{\ensuremath{~\mathrm{mod}~}}

\newcommand{\tend}[1]{\underset{#1}{\longrightarrow}}
\newcommand{\zqz}{\ensuremath{\Z/q \Z}}
\newcommand{\zqzc}{\ensuremath{\left(\Z/q \Z\right)^\times}}

\newcommand{\zpalph}{\ensuremath{\left(\Z/p^\alpha \Z\right)^\times}}

\makeatletter
\newcommand{\svdots}{%
	\vbox{\fontsize{\sf@size}{\sf@size pt}\linespread{0.3}\selectfont
		\kern0.2\baselineskip
		\hbox{.}\hbox{.}\hbox{.}%
		\kern0.1\baselineskip
	}%
}
\makeatother

 \theoremstyle{plain}

 \newtheorem{theo}{Theorem}
 \newtheorem{defi}{Definition}[section]
 \newtheorem{theobis}[defi]{Theorem}
 \newtheorem{propamoi}[theo]{Proposition}
 \newtheorem{lem}[defi]{Lemma}
 \newtheorem{prop}[defi]{Proposition}
 \newtheorem{cor}[defi]{Corollary}

 \theoremstyle{definition}
 
 \newtheorem{rem}[defi]{Remark}
 \newtheorem{exemple}[defi]{Example}

\titleformat
{\section} 
{\bfseries\Large} 
{\thesection.\ } 
{0ex} 
{} 
[] 

\titleformat
{\subsection} 
{\bfseries\large} 
{\thesubsection.\ } 
{0ex} 
{} 
[] 

\begin{document}
\renewcommand{\refname}{References}
\renewcommand{\contentsname}{Contents}
\renewcommand{\proofname}{Proof}
\begin{center}
	 \begin{minipage}{0.8\textwidth}
	 	\centering
	 	\Large \textbf{\textsc{Equidistribution of exponential sums indexed by a subgroup of fixed cardinality}}
	 \end{minipage}\\
 
 \vspace{1cm}
 Théo Untrau
\end{center}	 
\vspace{1cm}
\textsc{Abstract.} We consider families of exponential sums indexed by a subgroup of invertible classes modulo some prime power $q$. For fixed $d$, we restrict to moduli $q$ so that there is a unique subgroup of invertible classes modulo $q$ of order $d$. We study distribution properties of these families of sums as $q$ grows and we establish equidistribution results in some regions of the complex plane which are described as the image of a multi-dimensional torus via an explicit Laurent polynomial. In some cases, the region of equidistribution can be interpreted as the one delimited by a hypocycloid, or as a Minkowski sum of such regions.\\

\vspace{1cm}

\tableofcontents

\newpage 
\section{Introduction}

\subsection{Equidistribution of complete sums: the example of Kloosterman sums}

Let $q= p^\alpha$, where $p$ is an odd prime and $\alpha \in \Z_{ \geqslant 1}$. The classical Kloosterman sums are the real numbers defined by $$ \K_q(a,b) := \sum_{x \in \zqzc}^{} e\left( \frac{a x + b x^{-1}}{q}\right) $$
for any integers $a$ and $b$. Throughout this article, we use the notation $e(z)$ for $\exp(2i \pi z)$ and $x^{-1}$ for the inverse of $x \text{ modulo } q$.
These sums satisfy the bound\footnote{Here one really needs to assume that $p$ is an odd prime. When $q = 2^\alpha$ with $\alpha \geqslant 5$, the upper bound \eqref{Kqupperbound} needs to be replaced by $\left|\K_q(a,b)\right| \leqslant (2 \sqrt 2) \sqrt q$ (see the corrigendum \cite{corrigendum} to the article \cite{fouvrymichel}).}:
\begin{equation} \label{Kqupperbound}
\left|\K_q(a,b)\right| \leqslant 2 \sqrt q \quad \text{for all } a,b \in \zqzc,
\end{equation}
which is a consequence of Weil's work on the Riemann hypothesis for curves over finite fields when $\alpha =1$, and elementary computations when $\alpha \geqslant 2$ (see \cite[Corollary 1]{dubi}). This raises the question of the distribution of the sets of sums

$$\left\{ \frac{1}{\sqrt q} \K_q(a,b);  \  a,b \in \zqzc \right\}$$
in the interval $[-2,2]$ as $q$ goes to $+ \infty$. A result due to Katz asserts that the sets of sums\\ $\left\{  \frac{1}{\sqrt p}\K_p(a,1);  \  a \in \F_p^\times \right\}$ become equidistributed with respect to the Sato-Tate measure on $[-2,2]$:

$$\dd \mu_{\mathrm{ST}}(x) = \frac{1}{2\pi}\sqrt{4-x^2}\dd x $$

as $p \to + \infty$ through primes (see \cite[Example 13.6]{katz} for this specific statement). This relies on Deligne's equidistribution theorem and involves deep notions of algebraic geometry.\\

In the case where $q = p^\alpha$ is a non-trivial prime power (i.e. $\alpha \geqslant 2$), one can prove via elementary methods an equidistribution result for the sets $\left\{  \frac{1}{\sqrt q}\K_q(a,1);  \  a \in \zqzc \right\}$ as $q$ goes to infinity, see \cite[Remark 1.1]{dubi}. In this case, the measure with respect to which the sums become equidistributed is the measure $\mu$ defined as follows:

$$ \dd \mu(x) = \frac{1}{2} \delta_0(x) + \frac{1}{2 \pi} \frac{1}{\sqrt{4-x^2}} \dd x.$$
The following figure illustrates these two different behaviours. 
\begin{figure}[H]
	\centering
	\begin{subfigure}[t]{0.45\textwidth}
		\centering
		\includegraphics[width=\textwidth]{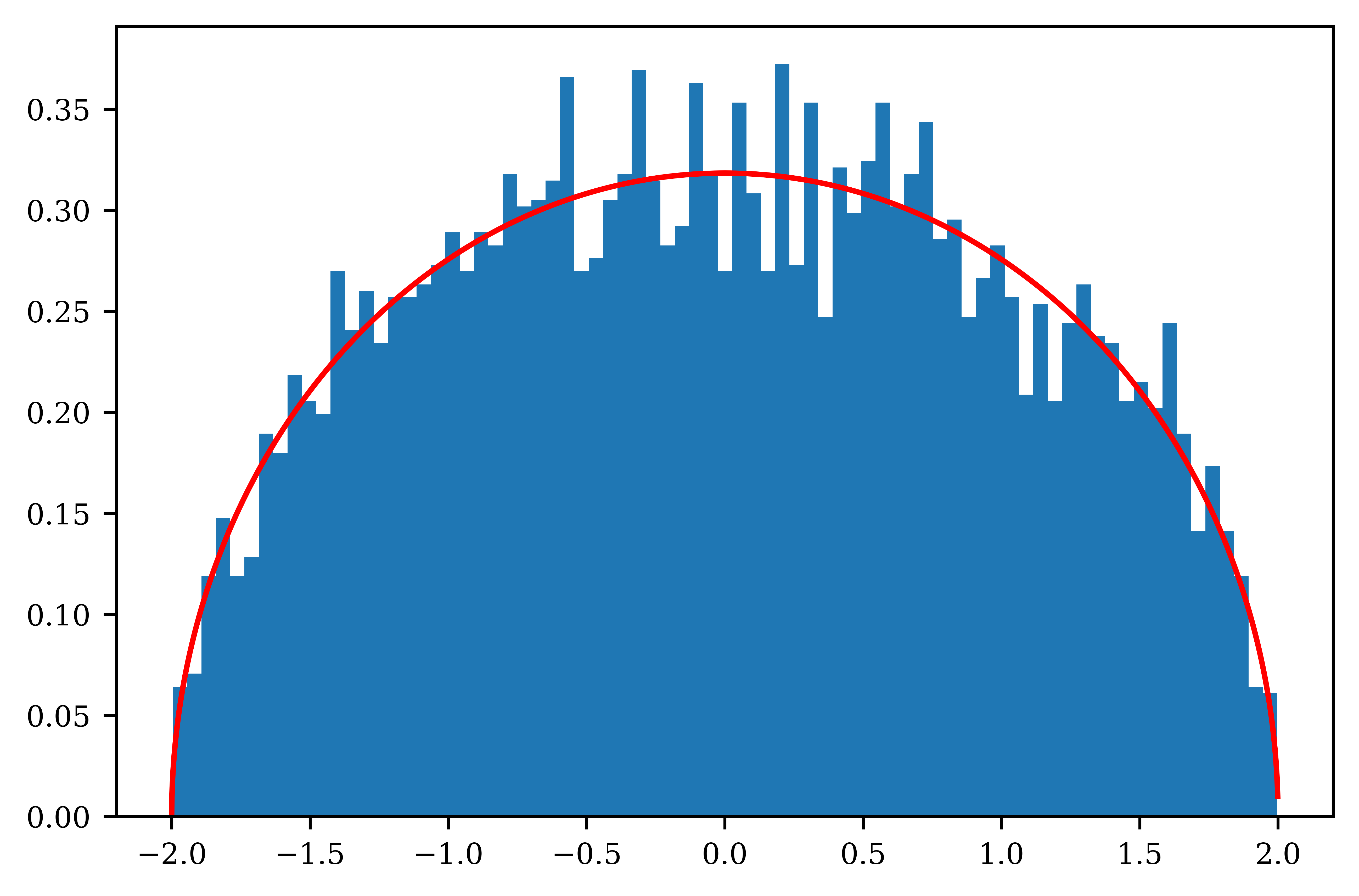}
		\caption{Distribution of the values $  \frac{1}{\sqrt{6007}}\K_{6007}(a,1)$ in $[-2,2]$ as $a$ ranges in $\F_{6007}^\times$. The red curve is the graph of $x \mapsto \frac{1}{2\pi}\sqrt{4-x^2}$.}
		\label{satotate}
	\end{subfigure}
	\hfill
	\begin{subfigure}[t]{0.45\textwidth}
		\centering
		\includegraphics[width=\textwidth]{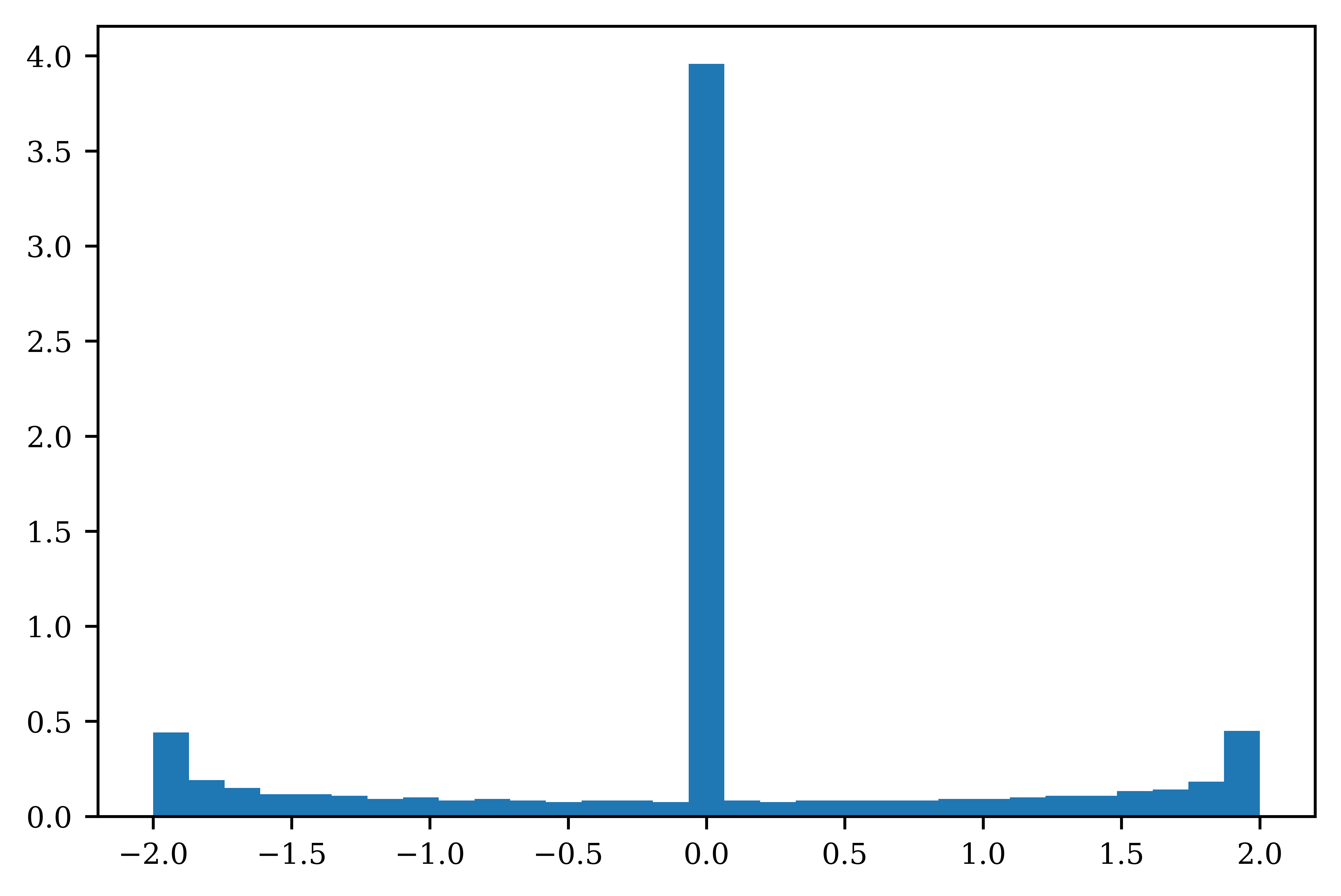}
		\caption{Distribution of the values $  \frac{1}{31}\K_{31^2}(a,1)$ in $[-2,2]$ as $a$ ranges in $(\Z/ 31^2 \Z)^\times$.}
		\label{diracenzero}
	\end{subfigure}
\caption{Distribution of normalized Kloosterman sums modulo a prime and modulo a prime power.}
\end{figure}

\subsection{Equidistribution of sums indexed by a subgroup}

The aim of this work is to study the question of the distribution of sums indexed by a subgroup of $\zqzc$. This question is motivated by the equidistribution results already known for complete sums, such as the ones presented in the previous section, as well as the appealing figures shown in the articles \cite{periods, visual} and \cite{menagerie}. In the latter, the authors fix an integer $d$ and introduce the \enquote{restricted} geometric sums:
\begin{equation} \label{xdegalungeom}
	\Sim_q(a,d) :=  \sum_{\substack{x \in \zqzc \\ x^d =1}}^{} e\left(\frac{ax}{q}\right)
\end{equation}
Then, the equidistribution of the sets $\left\{  \Sim_q(a,d);  \  a \in  \zqz \right\}$ as $q$ tends to infinity is investigated. In order to avoid degenerate cases in the index set of the sum defining $\Sim_q(a,d)$ and other sums in the remainder of this article, we make the following definition.

\begin{defi} \label{dadm}
	An integer $q$ will be called $d$-admissible if it is of the form $p^\alpha$ for some odd prime number $p$ congruent to $1$ modulo $d$, and some integer $\alpha \geqslant 1$. We denote by $\mathcal A_d$ the set of $d$-admissible integers.
\end{defi}

If $q$ is $d$-admissible, then the group $\zqzc$ has a unique subgroup of order $d$, explicitly described as $\{x \in \zqzc;\  x^d = 1 \}$. Thus, the sum in \eqref{xdegalungeom} can be interpreted as the one indexed by the unique subgroup of order $d$ of $\zqzc$. \\

In order to state the equidistribution result proved in \cite{menagerie, periods}, we need one last definition.

\begin{defi} \label{defgd}
	Let $d \geqslant 1$. For all $k \in \{ 0, \dots,d-1\}$, we denote by $(c_{j,k})_{0 \leqslant j < \varphi(d)}$ the coefficients of the remainder in the euclidean division of $X^k$ by $\phi_d$, the $d^{th}$ cyclotomic polynomial over $\Q$; precisely, these coefficients are defined by the property
	$$X^k \equiv \sum_{j=0}^{\varphi(d) -1} c_{j,k} X^j \mod \phi_d.$$
	Then, we define the Laurent polynomial
	
	$$\begin{array}{ccccc}
	g_d & : & \T^{\varphi(d)} & \to & \C \\
	&& (z_1, \dots, z_{\varphi(d)}) & \mapsto & \displaystyle \sum ^{d-1}_{k=0}\prod ^{\varphi \left( d\right) -1}_{j=0}z_{j+1}^{c_{j,k}} 
	\end{array}$$
\end{defi}

With these notations, the main theorem of \cite{menagerie, periods} on the asymptotic behaviour of sums of type \eqref{xdegalungeom} can be stated as follows. In \emph{loc. cit.} the theorem is stated as a density result, but the proof actually shows that equidistribution holds with respect to the appropriate pushforward measure.

\begin{theobis}[{\cite[Theorem 1]{menagerie} and {\cite[Theorem 6.3]{periods}}}] \label{thconnu1} Let $d \geqslant 1$. The sets \\
	$\left\{  \Sim_q(a,d);  \  a \in  \zqz \right\}$ become equidistributed in the image of $g_d$ with respect to the pushforward measure of the probability Haar measure $ \lambda$ on $\T^{\varphi(d)}$ via $g_d$, as $q$ goes to infinity among the $d$-admissible integers. In other words, for any continuous map $F\colon g_d\left(\T^{\varphi(d)}\right) \to \C$,
	$$\frac{1}{q} \sum_{a \in \zqz}^{} F\left(\Sim_q(a,d)\right) \tend{\substack{ q \to \infty \\ q \in \mathcal A_d}} \int_{\T^{\varphi(d)}} (F\circ g_d) \dd  \lambda .$$
\end{theobis}

Besides, it was proved in \cite[Theorem 7 and Theorem 10]{visual} that when $d$ is a prime number or $d=9$, the same equidistribution result holds for the sets of restricted Kloosterman sums $\left\{ \K_q(a,b, d); \ a,b \in (\zqz)^2 \right\}$, where
\begin{equation} \label{xdegalunkloos}
\K_q(a,b,d) :=  \sum_{\substack{x \in \zqzc \\ x^d =1}}^{} e\left(\frac{ax + b x^{-1}}{q}\right).
\end{equation}

Finally, for some specific values of the integer $d$ such as primes and prime powers, one can give a geometric interpretation of the image of $g_d$ in terms of hypocycloids.

\begin{defi} \label{defhypo}
	The $d$-cusp hypocycloid is the curve given by the image of:
	$$\begin{array}{ccc}
	\R &\mapsto &\C \\
	\theta & \mapsto & (d-1)\exp(i \theta) +  \exp((1-d) i \theta)
	\end{array}$$
	It is a curve described by a point of a circle of radius $1$ rolling inside a circle of radius $d$. 
\end{defi}

\begin{figure}[H]
	\centering
	\includegraphics[width= 0.8\textwidth]{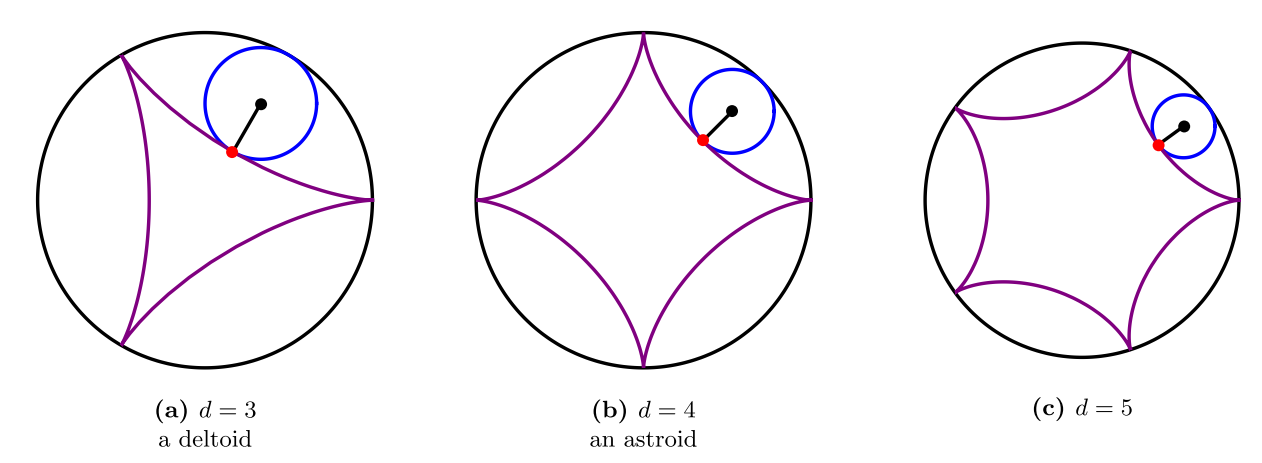}
	\caption{Some hypocycloids (image extracted from the article \cite{visual})}
\end{figure}

\begin{defi} \label{defHd}
	For all $d \geqslant 2$, we denote by $\H_d$ the closed region of the plane delimited by the $d$-cusp hypocycloid.
\end{defi}

Note that the $2$-cusp hypocycloid is just the interval $[-2,2]$, so it does not really enclose an area of the complex plane. Thus $\H_2$ is simply the interval $[-2,2]$ as well.\\

When $d$ is a prime, an explicit computation of $g_d$ leads to \cite[Proposition 1]{menagerie}, which states that the image of $g_d$ is the region $\H_d$. This yields a more concrete form of Theorem \ref{thconnu1}.

\begin{theobis}[{\cite[proof of Theorem 1.1]{periods}} and { \cite[Theorem 7 and p. 243,244]{visual}}] \label{thconnu2}
	Let $d$ be a prime number. Then the sets of sums $\left\{  \Sim_q(a,d);  \  a \in  \zqz \right\}$ become equidistributed in $\H_d$ with respect to the pushforward measure of the probability Haar measure on $\T^{d-1}$ via the map $$ g_d \colon (z_1, \dots, z_{d-1}) \mapsto z_1 + \dots + z_{d-1} + \frac{1}{z_1\cdots z_{d-1}}$$ as $q$ goes to infinity among the $d$-admissible integers. The same statement holds for the sets of sums \\$\left\{ \K_q(a,b, d); \ a,b \in (\zqz)^2 \right\}$.
\end{theobis}

The following picture illustrates the asymptotic behaviour predicted by this theorem in the case of Kloosterman sums.

\begin{figure}[H]
	\centering
	\begin{subfigure}[b]{0.32\textwidth}
		\centering
		\includegraphics[width=\textwidth]{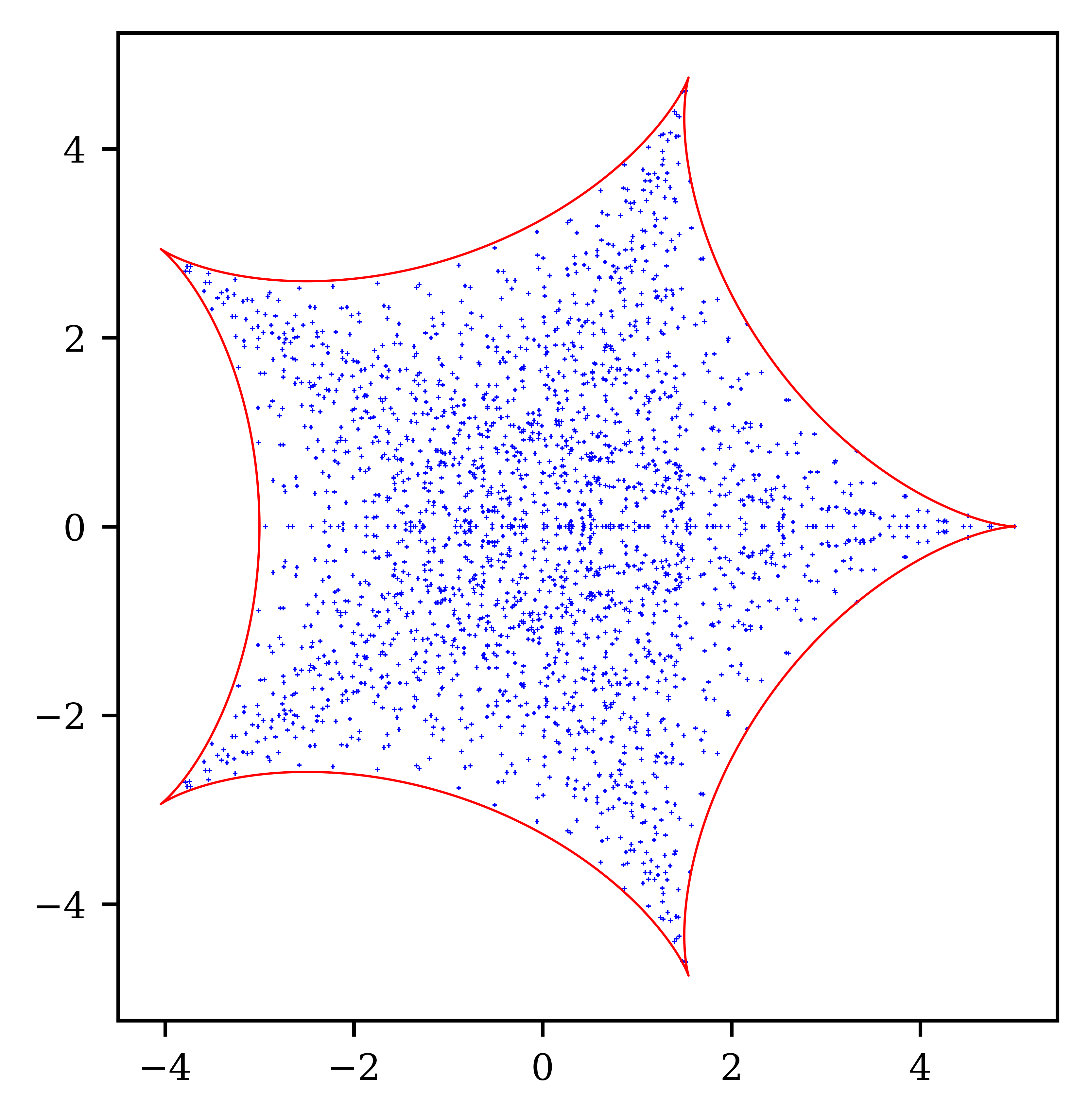}
		\caption{$q = 151$}
	\end{subfigure}
	\hfill
	\begin{subfigure}[b]{0.32\textwidth}
		\centering
		\includegraphics[width=\textwidth]{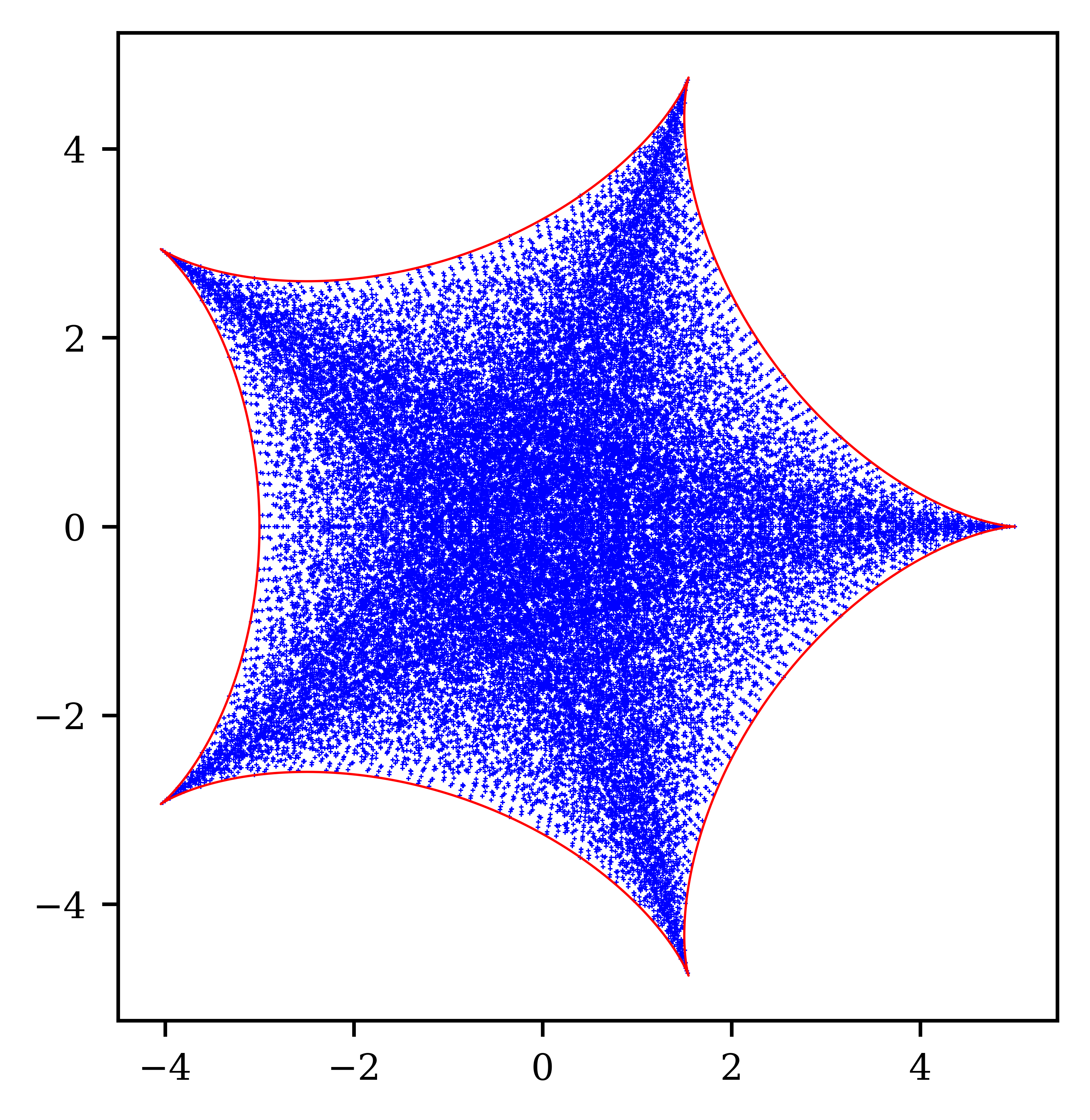}
		\caption{$q= 631$}
	\end{subfigure}
	\hfill
	\begin{subfigure}[b]{0.32\textwidth}
		\centering
		\includegraphics[width=\textwidth]{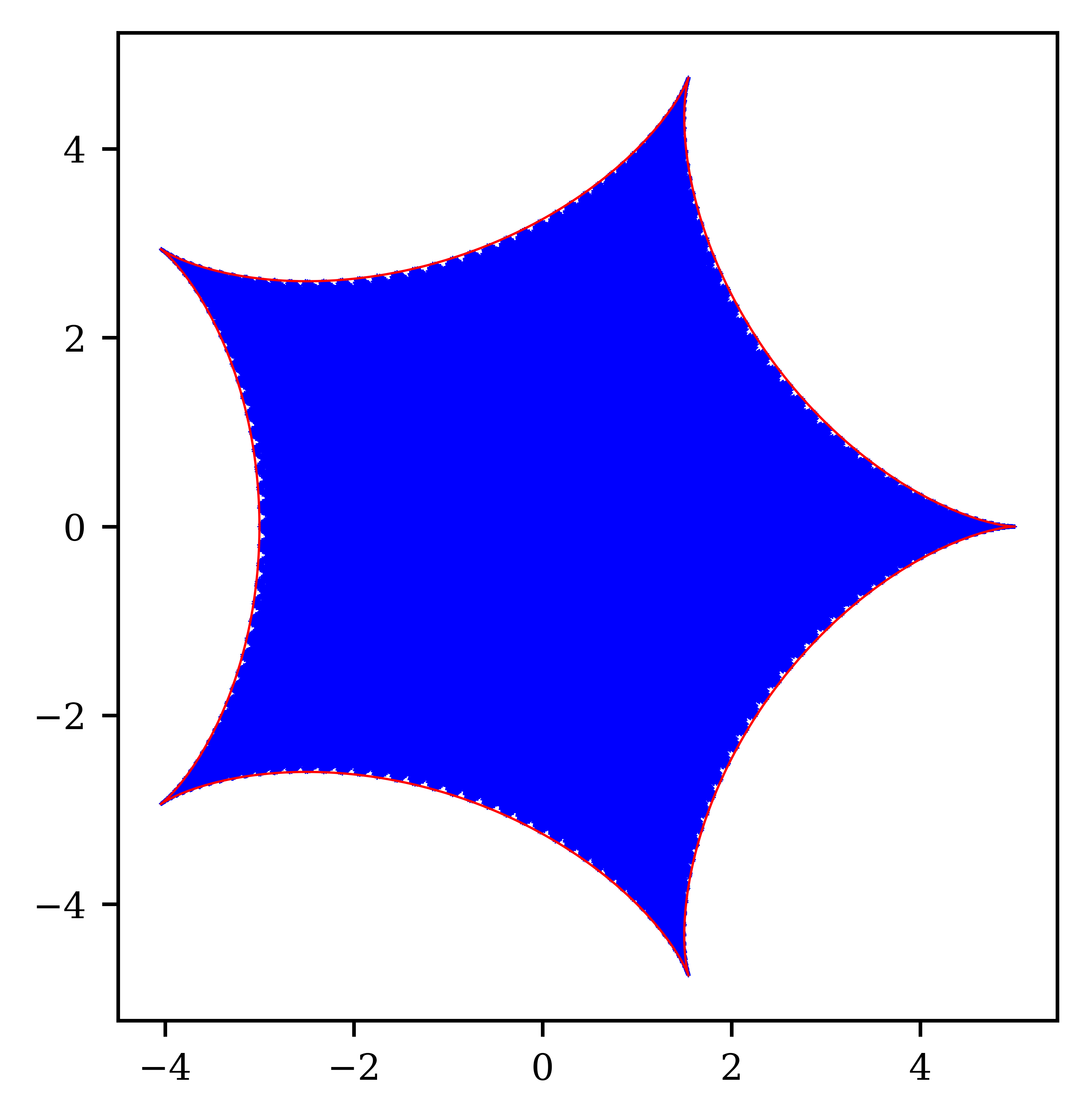}
		\caption{$q=3721 =61^2$}
	\end{subfigure}
	\caption{The sets $\left\{ \K_q(a,b, d); \ a,b \in (\zqz)^2 \right\}$ for $d= 5$ and three $5$-admissible values of $q$.}
	\label{3hypo5branches}
\end{figure}

The aim of this article is to generalize theorems \ref{thconnu1} and \ref{thconnu2} to more general families of exponential sums, and to study the question of restricting the parameters $a,b$ indexing the sums $\K_q(a,b,d)$, or generalizations of these, to certain specific subsets of $\left(\zqz\right)^2$, while preserving the equidistribution result. \\

Another motivation for studying exponential sums restricted to multiplicative subgroups comes from the article \cite{shkredov}, where considerations on sums over subgroups of $\F_p^\times$ lead to a new upper bound on Heilbronn's exponential sums. However, in the latter the size of the subgroups grows with $p$, so our problem will be quite different since we will be working with sums indexed by a subgroup of \textit{fixed cardinality}.

\subsection{Statement of the main result}
The study of the equidistribution of sets of sums of type  \eqref{xdegalungeom} and \eqref{xdegalunkloos} can be seen as a particular case of the following question: given a sequence $\left(\mathcal F_q\right)_{q \in \mathcal A_d}$ indexed by the $d$-admissible integers (Definition \ref{dadm}), where each $\mathcal F_q$ is a set of Laurent polynomials with coefficients in $\zqz$, what can be said about the distribution of the sets of sums
\begin{equation} \label{denomf} \left\{ \sum_{\substack{x \in \zqzc \\ x^d = 1}} e\left(\frac{f(x)}{q}\right); \ f \in \mathcal F_{q} \right\} 
\end{equation}
as $q$ goes to infinity among the $d$-admissible integers? In \eqref{xdegalungeom}, it is the case where $\mathcal F_q = \left\{ a X; \  a \in \zqz \right\}$ whereas \eqref{xdegalunkloos} corresponds to the case where $\mathcal F_q = \left\{ a X + \frac{b}{X}; \ (a,b) \in (\zqz)^2 \right\}$.
As Theorem \ref{thconnu2} shows, both cases surprisingly lead to the same regions of equidistribution, at least in the case where $d$ is a prime number. Thus, it is natural to ask whether these results extend to more general Laurent polynomials.\\

 Our main result (Theorem \ref{thprincipal}) generalizes these known cases. In order to state it we first define a few extra quantities.
 
 \begin{defi} \label{vectprem}
 	Let $d \geqslant 1$ be an integer, and let $\mm = (m_1, \dots, m_n) \in \Z^n$. We say that $\mm$ is coprime with $d$ if all the $m_i$ are coprime with $d$.
 \end{defi}

\begin{defi}
	Given $\mm = (m_1, \dots, m_n) \in \Z^n$ and $q \geqslant 1$, we denote by $\mathcal F_{\mm, q} $ the following set of Laurent polynomials with coefficients in $\zqz$:
	
	$$\mathcal F_{\mm, q} := \left\{ a_1 X^{m_1} + a_2 X^{m_2} + \dots +a_n X^{m_n}; \ (a_1,\dots, a_n) \in (\zqz)^n\right\}$$
\end{defi}

In case (b) of Theorem \ref{thprincipal} and Proposition \ref{propavecmyersonfaible}, we will see that the key argument which explains why sets of type \eqref{xdegalungeom} and \eqref{xdegalunkloos} become equidistributed in the same regions of the complex plane is that the corresponding $\mm$ is coprime with $d$, for any $d$. Indeed, in case \eqref{xdegalungeom} we have $\mm = (1) \in \Z$ and in case \eqref{xdegalunkloos} we have $\mm = (1, -1) \in \Z^2$. \\
Even though we will also treat the case where $\mm$ is not coprime with $d$, this observation is the starting point that led us to the generalizations that we prove in the current work. Precisely we focus on the distribution of the following sets of sums:

\begin{equation} \label{formegen}
	\left\{ \sum_{\substack{x \in (\zqz)^\times \\ x^d = 1}} e\left(\frac{a_1 x^{m_1} + \dots + a_n x^{m_n}}{q}\right); \ (a_{1}, \dots ,a_{n}) \in (\zqz)^n \right\}
\end{equation}
 In other words, these are sets of exponential sums of the form \eqref{denomf} with $\mathcal F_q$ equal to $\mathcal F_{\mm,q}$ for some $\mm \in \Z^n$. \\
In fact we prove a more general result by showing that it is possible to impose strong restrictions on the set of parameters and still obtain equidistribution. Our main result is indeed concerned with sets of sums of the form
\begin{equation} \label{formeplusgen}
	\left\{ \sum_{\substack{x \in (\zqz)^\times \\ x^d = 1}} e\left(\frac{a_1 x^{m_1} + \dots + a_n x^{m_n}}{q}\right); \ (a_1, \dots, a_n) \in H^{(1)}_q \times \dots \times H^{(n)}_q \right\},
\end{equation}
where the $H_q^{(i)}$ are sufficiently large subgroups of $\zqzc$.\\

We finally define the relevant Laurent polynomials that will come into play in the description of the region of equidistribution of sets of type \eqref{formegen} and \eqref{formeplusgen} in the case where $\mm$ is not coprime with $d$.

\begin{defi}  \label{deffd} Let $d \geqslant 1$ and let $\mm =(m_1, \dots, m_n) \in \Z^n$. For all $i \in \{1, \dots, n\}$, we denote by $$d_i := \frac{d}{(d,m_i)}$$
	and by $\left(c_{j,k}^{(i)}\right)_{0 \leqslant j < \varphi(d_i)}$ the coefficients that appear in the reduction modulo $\phi_{d_i}$ of $X^k$ for each $k$ in $\{0, \dots, d-1\}$. In other words, these are the unique integers such that:
	$$\forall k \in \{0, \dots, d-1\}, \qquad X^k \equiv  \sum_{j=0}^{\varphi(d_i)-1} c_{j,k}^{(i)} X^j \mod \phi_{d_i}.$$	
	
	Then we define the Laurent polynomial $f_{d, \mm}$ as follows: 
	
	\begin{equation} \label{eqfdm}
	\begin{array}{ccccc}
	f_{d, \mm} & : & \T^{\varphi(d_1)+ \dots \varphi(d_n)} & \to & \C \\
	&& ((z_{1,j})_{0 \leqslant j < \varphi(d_1)}, \dots, (z_{n,j})_{0 \leqslant j < \varphi(d_n)}) & \mapsto & \displaystyle \sum ^{d-1}_{k=0} \prod_{i=1}^{n} \prod^{\varphi(d_i) -1}_{j=0} z_{i,j}^{c_{j,k}^{(i)}} \end{array}
	\end{equation}
\end{defi}

We can now give the statement of the main result. 

\begin{theo} \label{thprincipal}
	Let $d \geqslant 1$ be an integer and let $\mm=(m_1, \dots, m_n) \in \Z^n$. For all $d$-admissible integer $q$, we fix subgroups $H^{(1)}_q , \dots ,H^{(n)}_q$ of $\zqzc$. Then we have the following equidistribution results:
	\begin{enumerate}[label=$\mathrm{(\alph*)}$]	
		
		\item	\emph{The general case.} \\
		If there exists $\delta > 0$ such that the subgroups $H^{(1)}_q , \dots ,H^{(n)}_q$ satisfy the growth condition:
		\begin{equation} \label{growth1}
		\forall i \in \{1, \dots,n\}, \quad |H_q^{(i)}| \geqslant q^{\delta},
		\end{equation}
		
		then the sets \eqref{formeplusgen} become equidistributed in the image of the Laurent polynomial $f_{d, \mm}$ (Definition $\ref{deffd}$) with respect to the pushforward measure via $f_{d, \mm}$ of the probability Haar measure $\lambda$ on $\T^{\varphi(d_1) +\dots + \varphi(d_n)}$, as $q$ goes to infinity among the $d$-admissible integers. In other words, if we denote by $\mathcal I_{d, \mm}$ the image of $f_{d, \mm}$ and by $\mu := (f_{d, \mm})_* \lambda$, then for all continuous function $F \colon \mathcal I_{d,  \mm}  \to \C$,
		 	$$ \frac{1}{\prod_{i= 1}^{n}|H_{q}^{(i)}|}  \sum_ {a_1\in H_{q}^{(1)}} \cdots \sum_{a_n \in H_q^{(n)}}  F\left( \sum_{\substack{x \in (\zqz)^\times \\ x^d = 1}} e\left(\frac{a_1 x^{m_1} + \dots + a_n x^{m_n}}{q}\right)\right) \tend{\substack{ q \to \infty \\ q \in \mathcal A_d}} \int_{\mathcal I_{d, \mm}}^{} F \dd \mu.$$
		
		\item \emph{When $\mm$ is coprime with $d$.} \\
	If there exists $\delta > 0$ such that the subgroups $H^{(1)}_q , \dots ,H^{(n)}_q$ satisfy the growth condition:
		\begin{equation} \label{growth2}
		\forall q \in \mathcal A_d,\  \exists i \in \{1, \dots,n\}, \quad |H_q^{(i)}| \geqslant q^{\delta},
		\end{equation}
		then the sets \eqref{formeplusgen} become equidistributed in the image of the Laurent polynomial $g_d$ (Definition $\ref{defgd}$) with respect to the pushforward measure via $g_d$ of the probability Haar measure on $\T^{\varphi(d)}$, as $q$ goes to infinity among the $d$-admissible integers.
		
	\end{enumerate}
\end{theo}
For instance, if one takes $\mm = (1,-1)$, the second case of this theorem states that the sets 
 \begin{equation} \label{kloossubgroups}
 	\left\{\K_q(a,b,d);  \ (a,b) \in H_q^{(1)} \times H_q^{(2)} \right\}
 \end{equation}
 satisfy the same equidistribution result as the sets of Figure \ref{3hypo5branches}, as soon as the $H_q^{(i)}$ satisfy the growth condition \eqref{growth2}. In other words, restricting the parameters $a,b$ to large enough multiplicative subgroups does not introduce any bias in the distribution of the restricted Kloosterman sums, and still ensures equidistribution with respect to the same measure as in Theorem \ref{thconnu2}. We give an illustration of this fact in section \ref{illusthprincipal}.

\begin{rem}
	We will also discuss the possibility to fix some of the parameters, while letting the other vary. For instance, this means that we will consider sets of the form \eqref{formeplusgen} but with the condition $(a_1, \dots, a_n) \in H^{(1)}_q \times \dots \times H^{(n)}_q$ replaced by
	$(a_{i_1}, \dots,a_{i_s}) \in H_q^{(i_1)} \times \dots \times H_q^{(i_s)}$ for some $s <n$, while the other parameters $a_j$ are fixed integers (see Remark \ref{fix}).
\end{rem}

Finally, the equidistribution result of Theorem \ref{thprincipal}, concerning sets of type \eqref{formeplusgen}, admits an analogue for sets of type \eqref{formegen}, via a simple adaptation of the proof. Precisely, we will obtain the following proposition, which generalizes \cite[Theorem 7]{visual} and \cite[Theorem 6.3]{periods}.

\begin{propamoi}  \label{propavecmyersonfaible}
	Let $d \geqslant 1$, and let $\mm = (m_1, \dots, m_n) \in \Z^n$.
\begin{enumerate}[label=$\mathrm{(\alph*)}$]

\item \emph{The general case.} \\
The sets of sums \eqref{formegen} become equidistributed in the image of the Laurent polynomial $f_{d, \mm}$ (from Definition $\ref{deffd}$) with respect to the pushforward measure via $f_{d, \mm}$ of the probability Haar measure on $\T^{\varphi(d_1) + \dots + \varphi(d_n)}$, as $q$ tends to infinity among the $d$-admissible integers.

\item  \emph{When $\mm$ is coprime with $d$.} \\
 Let $s \in \{1, \dots, n\}$ and let $\{i_1, \dots, i_s\} \subseteq \{1, \dots, n\}$. We fix $n-s$ integers $a_i$ for $i \in \{1, \dots, n\} \setminus \{i_1, \dots, i_s\}$. Then the sets of sums

$$	\left\{ \sum_{\substack{x \in (\zqz)^\times \\ x^d = 1}} e\left(\frac{a_1 x^{m_1} + \dots + a_n x^{m_n}}{q}\right); \ (a_{i_1}, \dots ,a_{i_s}) \in (\zqz)^s  \right\}$$
become equidistributed in the image of $g_d$ (with respect to the pushforward measure via $g_d$ of the probability Haar measure on $\T^{\varphi(d)}$) as $q$ goes to infinity among the $d$-admissible integers.
\end{enumerate}
\end{propamoi}
If one takes $\mm$ to be equal to $(1) \in \Z$ or $(1, -1) \in \Z^2$, then the second case of this proposition allows one to recover Theorem \ref{thconnu1} as well as Theorem \ref{thconnu2} extended to values of $d$ which are not prime. In particular, the asymptotic behaviour shown in Figure \ref{3hypo5branches} is an illustration of Proposition \ref{propavecmyersonfaible} (b) in the case of Kloosterman sums. We give other examples of application in section \ref{illusprop}.
\begin{rem}
	The Laurent polynomial $g_d$ does not depend on $\mm$, as long as $\mm$ is coprime with $d$. This implies that the region of equidistribution almost does not depend on the shape of the numerators in the exponentials: it will be the same for any $\mm$ coprime with $d$. This explains why \cite[Theorem 7]{visual} and \cite[Theorem 6.3]{periods} give rise to the same kind of figures, and this leads to many other examples. Similarly, the Laurent polynomial $f_{d, \mm}$ only depends on $\mm$ through the list of the gcd's $(d,m_i)$.
\end{rem}

\subsection{Strategy of the proof of Theorem \ref{thprincipal}.} {The first step consists in reducing both cases to two statements about the equidistribution modulo $1$ of some sets of arithmetic nature: propositions \ref{equimyersoncase1} and \ref{equimyersoncase2}. These two propositions can be seen as a generalization of Myerson's lemma\footnote{The name comes from the fact that this is an adaptation of an argument that is used in the proof of \cite[Theorem 12]{myerson}.}, which asserts that the sets \begin{equation} \label{formeaveczqz1}
\left\{\frac{a}{q} \left(1, w_q, \dots , w_q^{\varphi(d)-1}\right); \ a \in \zqz\right\},
\end{equation} where $w_q$ is a primitive $d$-th root of unity modulo $q$, become equidistributed modulo $1$ as $q$ goes to infinity among the $d$-admissible integers. This lemma is proved using a version of Weyl's equidistribution criterion and thus reduces to the following exponential sum estimate:
\begin{lem}[Myerson's lemma, {\cite[Lemma 6.2]{periods}}] \label{myerson} Let $d \geqslant 1$ be an integer, and let $f \in \Z[X] \setminus \{0\}$ be a polynomial of degree strictly less than $\varphi(d)$. Then there exists an integer $m_f$ such that for all $d$-admissible integer $q$ such that $q > m_f$, for any element $w_q$ of order $d$ in $\zqzc$,
	$$\sum_{a\in \zqz}^{} e\left( \frac{f( w_q)}{q} a\right) =0.$$
\end{lem} }

Propositions \ref{equimyersoncase1} and \ref{equimyersoncase2} essentially amount to generalizing the equidistribution of sets of type \eqref{formeaveczqz1} to sets of the form 
\begin{equation} \label{formeavechq1}
\left\{\frac{a}{q} \left(1, w_q, \dots , w_q^{\varphi(d)-1}\right); \ a \in H_q\right\},
\end{equation}
where $H_q$ is a large enough subgroup of $\zqzc$. Precisely, a particular case of Proposition \ref{equimyersoncase2} is the following corollary which generalizes \cite[Lemma 6.2]{periods}.

\begin{cor} \label{coromyerson} Let $d \geqslant 1$ and let $\delta > 0$. For all $q \in \mathcal A_d$, let $w_q$ be an element of order $d$ in $\zqzc$. For each of these values of $q$, we also fix a subgroup $H_q$ of $\zqzc$. If the following growth condition is satisfied:  $$ |H_q| \geqslant q^\delta,$$ 
	then the sets \eqref{formeavechq1} become equidistributed modulo $1$ as $q$ tends to infinity among the $d$-admissible integers.
\end{cor}

The crucial input in order to obtain convergence towards zero when applying Weyl's criterion is the following exponential sum estimate, which relies on a deep result due to Bourgain and stated in Theorem \ref{bourgain}. 
\begin{prop} \label{myersongen}
	Let $d \geqslant 1$ and let $f\in \Z[X] \setminus \{0\}$ be a polynomial of degree strictly less than $\varphi(d)$. \\
	Let $\delta > 0$. Then, there exists $\varepsilon = \varepsilon(\delta) > 0$, depending only on $\delta$, such that for all $d$-admissible integer $q$ large enough, for all subgroup $H_q$ of $\zqzc$ satisfying $|H_q| \geqslant q^{\delta}$, and for any element $w_q$ of order $d$ inside $\zqzc$, we have
	\begin{equation} \label{sommefwq}
	\left| \sum_{a \in H_{q}}^{} e \left( \frac{af(w_{q})}{{q}}\right) \right| \ll_{f, \delta} \frac{|H_q|}{q^{\varepsilon}} \cdot 
	\end{equation}	
\end{prop}
This estimate allows us to complete the proofs of propositions \ref{equimyersoncase1} and \ref{equimyersoncase2}, thus proving Theorem \ref{thprincipal}. Let us stress that the application of Bourgain's theorem is not straightforward, and the second main ingredient in the proof of Proposition \ref{myersongen} is a good understanding of the $p$-adic valuation of $f(w_q)$ (see Proposition \ref{lembezout}).

\paragraph{Structure of the paper.} Section \ref{grandesectionreduc} is devoted to reducing the proof of the main result to statements about equidistribution modulo $1$. In Section \ref{secmyerson}, we prove the key exponential sum estimate (Proposition \ref{myersongen}) which allows us to obtain the convergence towards zero in the applications of Weyl's criterion of Section \ref{seclemmamod1}. In the latter, we establish the needed properties of equidistribution modulo $1$ and conclude the proof of Theorem \ref{thprincipal}. Finally in Section \ref{notables}, we give examples and illustrations.

\subsection{Notations} \label{secnot}
\begin{itemize}
	\setlength\itemsep{0em}
	\item The number of elements of a finite set $X$ is denoted by $|X|$.
	
	\item If $a,b \in \Z$, we denote by $(a,b)$ their gcd (greatest (positive) common divisor).
	
	\item If $a \in \Z$ and $p$ is a prime number, we denote by $v_p(a)$ the $p$-adic valuation of $a$.
	
	 \item If $d$ is a positive integer, $\phi_d$ denotes the $d^{\text{th}}$ cyclotomic polynomial over $\Q$ and $\varphi(d)$ its degree.
	
	\item $\mu$ denotes the Möbius function.
	
	\item If $x \in \R$, we denote by $\{x\} := x - \lfloor x \rfloor $ its fractional part. If $\xx = (x_1, \dots, x_m) \in \R^m$,  we denote by $\{ \xx \} := (\{x_1\}, \dots,\{x_m\} )$ the fractional part of $\xx$ taken componentwise.
	
	\item Let $(X, \mathscr A)$ and $(Y, \mathscr B)$ be two measurable spaces, and let $\lambda$ be a measure on the former. If $f \colon X \to Y$ is $(\mathscr A, \mathscr B)$-measurable, then we denote by $f_*\lambda$ the pushforward measure of $\lambda$ via $f$. It is defined as the measure on $(Y, \mathscr B)$ such that $(f_*\lambda)(B) = \lambda(f^{-1}(B))$ for all $B \in \mathscr B$.
	
	\item $\T$ denotes the group of complex numbers of modulus $1$.
\end{itemize}
The Jupyter Notebook which was written to obtain most of the figures of this article is available in html format at the URL: \url{http://perso.eleves.ens-rennes.fr/people/theo.untrau/sumssubgroups}\\

\textbf{Acknowledgements.} The author would like to thank Guillaume Ricotta and Florent Jouve for suggesting this question, and for the helpful discussions all along this work. We also thank \'Etienne Fouvry and Emmanuel Kowalski for their comments on a preliminary version of the article. The pictures were made with the open-source software \texttt{sagemath}: \cite{sagemath}.

\section{Reduction to statements on equidistribution modulo $1$}
\label{grandesectionreduc}

In this section, we prove that the two cases of Theorem \ref{thprincipal} are implied by two lemmas on the equidistribution modulo $1$ of certain sequences of sets of arithmetic nature. The idea is that the exponential sums we are considering, which are sums of $d$ particular roots of unity, can in fact be expressed as a Laurent polynomial in a smaller number of roots to unity. We start this section by stating a lemma which is the main ingredient to perform this reduction.

\subsection{Reduction modulo prime powers of cyclotomic polynomials}

\label{secreducmod1}

\begin{lem} \label{affirmation}
	Let $d \geqslant 1$ be an integer and let $q = p^\alpha$ be a $d$-admissible integer. Let $x \in \zqzc$ be an element of order $d$. Then we have: $$\overline \phi_d (x) = 0 \text{ in } \zqz$$
	where $\overline \phi_d$ stands for the reduction of $\phi_d$ modulo $q$.
\end{lem}

\begin{proof}
	We consider the polynomial $P(X) := X^{d} -1$, seen as an element in $\Z_p[X]$, where $\Z_p$ is the ring of $p$-adic integers. Let $\tilde x$ be a lift in $\Z$ of the class $x$ modulo $q$. Then we have
	$$P(\tilde x) \equiv 0 \mod q$$
	since $x$ has order $d$. Therefore $ |P(\tilde x) |_p \leqslant \frac{1}{p^\alpha}$, where we denoted by $|\cdot |_p$ the standard $p$-adic absolute value on the field of $p$-adic numbers $\Q_p$.
	On the other hand, we have $P'(\tilde x) = d\tilde x^{d-1}$, which has $p$-adic valuation zero since $(d,p) = 1$ (because $d$ divides $p-1$) and $(\tilde x,p)= 1$ since $x$ is invertible modulo $p^\alpha$. Thus, $|P'(\tilde x) |_p = 1$ and so:		
	$$  |P(\tilde x) |_p \leqslant \frac{1}{p^\alpha} =\frac{1}{p^\alpha} |P'(\tilde x) |^2_p$$ 
	Therefore, by Hensel's lemma (see \cite[chapter II, appendix C]{CetF}) there exists a unique $ \tilde z\in \Z_p$ such that 
	\begin{equation} \label{hensel} 
	\begin{cases}
	P(\tilde z) = 0 \\
	|\tilde z -\tilde x|_p \leqslant \frac{1}{p^\alpha}
	\end{cases}
	\end{equation}
We deduce that:
	\begin{equation} \label{prodphi_d}
	0 = \tilde  z^d - 1 = \prod_{m \mid d} \phi_m(\tilde z) \quad \text{in } \Z_p
	\end{equation}
	Now since $\Z_p$ is an integral domain, at least one of the factors $\phi_m( \tilde z)$ must be zero.\\
	Assume for a contradiction that this happens for an $m$ which is not equal to $d$. Then this would imply that $\tilde z^m = 1 $ in $\Z_p$, hence:
	
	$$ |\tilde x^m - 1|_p =|\tilde x^m - \tilde z^m |_p \leqslant |\tilde x - \tilde z|_p \leqslant \frac{1}{p^\alpha}$$ by the second condition in \eqref{hensel}. Thus, $\tilde x^m \equiv 1 \mod p^\alpha$ for an $m < d$, contradicting the fact that $x$ has order exactly $d$ in $ \zpalph$. Therefore, in the product \eqref{prodphi_d}, it is the term $\phi_d(\tilde z)$ which equals zero. Now, since $|\tilde x - \tilde z|_p \leqslant \frac{1}{p^\alpha}$ we have: $$|\phi_d(\tilde x) |_p =|\phi_d( \tilde x) - \phi_d(\tilde z) |_p \leqslant \frac{1}{p^\alpha}$$ and this is equivalent to $\phi_d(\tilde x) \equiv 0 \mod p^\alpha$, that is: $\overline \phi_d(x) = 0$ in $\Z/p^\alpha \Z$.
	
\end{proof}
In the remainder of this section, we state two lemmas on the equidistribution modulo $1$ of some particular sets, and prove that they imply Theorem \ref{thprincipal}.

\subsection{Reduction step for the main result: case (a)}
\label{secreduction}
First, let us state the proposition which will turn out to imply case (a) of Theorem \ref{thprincipal}.

\begin{prop}[Equidistribution modulo $1$ case (a)] \label{equimyersoncase1} Let $d \geqslant 1$ be an integer and $\mm =(m_1, \dots, m_n) \in \Z^n$. For all $i \in \{1, \dots, n\}$, the notation $d_i$ stands for $\frac{d}{(d,m_i)}$. Let $\delta > 0$. For all $d$-admissible integer $q$, let $w_q$ be an element of order $d$ in $\zqzc$. For each such $q$, we also choose subgroups of $\zqzc$: $H^{(1)}_q , \dots , H^{(n)}_q$. Then, provided the subgroups satisfy the growth conditions:
	$$\forall i \in \{1, \dots,n\}, \qquad |H_{q}^{(i)}| \geqslant q^\delta, $$
	the sets of $(\varphi(d_1) + \dots + \varphi(d_n))$-tuples
	
	$$\bigg \{  \overbrace{\left( \left(\frac{a_1 ( w_q^{m_1})^j}{q}\right)_{0 \leqslant j <  \varphi(d_1)} , \dots, \left(\frac{a_n (w_q^{m_n})^j}{q}\right)_{0 \leqslant j < \varphi(d_n)} \right)}^{=: \xx(a_1, \dots, a_n,q)}; \quad (a_1, \dots, a_n) \in H^{(1)}_q \times \dots \times H^{(n)}_q \bigg \}
$$	

	become equidistributed modulo $1$ as $q$ goes to infinity among the $d$-admissible integers. In other words, for any continuous map $G  \colon [0,1]^d \to \C$, 
	$$ \frac{1}{\prod_{i = 1}^n |H_q^{(i)}|} \sum_{\substack{a_1 \in H^{(1)}_q \\ \svdots \\ a_n \in H_q^{(n)}}} G(\{ \xx(a_1, \dots, a_n,q)\}) \tend{\substack{q \to \infty \\  q \in \mathcal A_d}} \int_{[0,1]^d} G(x_1, \dots, x_d) \dd x_1 \dots \dd x_d.$$ 
\end{prop}

\begin{proof}
See section	\ref{secpreuveequimyersoncase1}.
\end{proof}

\begin{rem} \label{abusdenot}
	There is a little abuse of notation here, since the $a_i$ and $w_q$ are classes modulo $q$, and the fractions above may depend on the choice of a representative. However, since we are only interested in the distribution properties \textit{modulo} $1$, we sometimes allow ourselves to keep on writing classes modulo $q$ at the numerator of fractions with denominator equal to $q$.
\end{rem}

This lemma on equidistribution modulo $1$ translates into a result on equidistribution of exponential sums restricted to a subgroup via the Laurent polynomials $f_{d, \mm}$ introduced in equation \eqref{eqfdm} of Definition \ref{deffd}.

\begin{prop} \label{reducmod1case1}
	Proposition $\ref{equimyersoncase1}$ implies case $\mathrm{(a)}$ of Theorem $\mathrm{\ref{thprincipal}}$.
\end{prop}

\begin{proof}
	Let $d\geqslant 1$ and $\mm = (m_1, \dots, m_n) \in \Z^n$. Let us denote by $\mathcal I_{d, \mm} := f_{d, \mm}\left(\T^{\varphi(d_1) +\dots + \varphi(d_n)} \right)$ and by $\mu := (f_{d, \mm})_*\lambda$ the pushforward measure via $f_{d, \mm}$ of the probability Haar measure $\lambda$ on $\T^{\varphi(d_1) +\dots + \varphi(d_n)}$. \\
	For all $d$-admissible integer $q$, let $$Y_{\mm,q} :=  H^{(1)}_{q} \times \dots \times H^{(n)}_{q}$$
	be a product of subgroups of $\zqzc$ as in case (a) of Theorem \ref{thprincipal} (that is: satisfying the growth condition \eqref{growth1}) and let $\theta_{\mm, q} \colon  Y_{\mm,q} \to \C $ be the map defined by 
\begin{equation} \label{defthetamq}
(a_1, \dots, a_n)  \mapsto  \sum_{\substack{x \in (\zqz)^\times \\ x^d = 1}} e\left(\frac{a_1 x^{m_1} + \dots + a_n x^{m_n}}{q}\right).
\end{equation}	
We want to show that, assuming Proposition \ref{equimyersoncase1}, we have
\begin{equation} \label{convintegrales}
	\frac{1}{|Y_{\mm,q}|}  \sum_{(a_1, \dots ,a_n) \in Y_{\mm,q}}  F\left(\theta_{\mm,q} \left( a_1, \dots, a_n\right) \right) \tend{\substack{q \to \infty \\ q \in \mathcal A_d}} \int_{\mathcal I_{d, \mm}}^{} F \dd \mu
\end{equation}
for any $F\colon \mathcal I_{d, \mm} \to \C$ continuous.\\

Let $w_q$ be an element of order $d$ in $(\zqz)^\times$. Then the unique subgroup of order $d$ inside $\zqzc$ can be described as $\left\{ w_q^k; \ k \in \{0, \dots, d-1\}\right\}$, so that for all $(a_1, \dots, a_n) \in Y_{ \mm,q},$
$$\theta_{\mm,q}\left( a_1, \dots, a_n\right) =\sum ^{d-1}_{k=0} \prod_{i=1}^{n} e\left(\frac{a_i (w_q^{m_i})^k}{q}\right).$$
	
Now,	
	$$ \forall i \in \{1, \dots, n\},\ \forall k \in \{0,\dots,d-1\}, \qquad (w_q^{m_i})^k = \sum_{j= 0}^{\varphi(d_i) - 1} c_{j,k}^{(i)} (w_q^{m_i})^j.$$
This comes from using lemma \ref{affirmation} after having evaluated the polynomial congruences defining the $c_{j,k}^{(i)}$ (see Definition \ref{deffd}) at $w_q^{m_i}$. The lemma applies since $w_q^{m_i}$ has order $d_i$ in $\zqzc$. Replacing this in the expression of $\theta_{\mm,q}\left( a_1, \dots, a_n\right)$ obtained above, we get:
	
$$\theta_{\mm,q}\left( a_1, \dots, a_n\right)
		= \sum ^{d-1}_{k=0} \prod_{i=1}^{n} \prod_{j=0}^{\varphi(d_i)-1} e\left(\frac{a_i (w_q^{m_i})^j}{q}\right)^{c_{j,k}^{(i)}}.$$
	
		Therefore, if we define for all $i \in \{1, \dots, n\}$ and for all $j \in \{ 0, \dots, \varphi(d_i)-1 \}$, 
	
	\begin{equation*} \label{introzij}
		z_{i,j} = z_{i,j}(a_1, \dots, a_n, q) := e\left(\frac{a_i (w_q^{m_i})^j}{q}\right)  
	\end{equation*}
	
	we have \begin{equation} \label{fdzij}
		\theta_{\mm,q} \left( a_1, \dots, a_n\right) = f_{d, \mm} \left((z_{1,j})_{0 \leqslant j < \varphi(d_1)}, \dots, (z_{n,j})_{0 \leqslant j< \varphi(d_n)}\right)
	\end{equation}

	with the Laurent polynomial $f_{d, \mm}$ from Definition \ref{deffd}, and the $z_{i,j}$ being elements of $\T$.  This already shows that $\theta_{\mm,q} \left( a_1, \dots, a_n\right) $ belongs to the image of $f_{d, \mm}$. \\

 Now, let us prove that the equidistribution statement. Let $F\colon \mathcal I_{d, \mm} \to \C$ be a continuous function. Thanks to \eqref{fdzij}, the left-hand side of \eqref{convintegrales} may be rewritten as

 $$\frac{1}{|Y_{\mm,q}|}  \sum_{(a_1, \dots, a_n) \in Y_{\mm,q}}  F\left( f_{d, \mm}\left((z_{1,j})_{0 \leqslant j < \varphi(d_1)}, \dots, (z_{n,j})_{0 \leqslant j < \varphi(d_n)}\right) \right)$$

and this converges to 

$$ \int_{\T^{\varphi(d_1) +\dots + \varphi(d_n)}}^{}\left( F \circ f_{d, \mm} \right) \dd \lambda =  \int_{\mathcal I_{d, \mm}}^{} F \dd \mu$$
by proposition \ref{equimyersoncase1} and because $\mu = (f_{d, \mm})_*\lambda$. This finishes the proof of \eqref{convintegrales}. Thus, Theorem \ref{thprincipal} (a) indeed follows from Proposition \ref{equimyersoncase1}.
\end{proof}

\subsection{Reduction step for the main result: case (b)}
\label{secreduccase2}

As in the previous case, let us begin with the statement of the proposition which will imply case (b) of Theorem \ref{thprincipal}.

\begin{prop}[Equidistribution modulo $1$ case (b)] \label{equimyersoncase2} Let $d \geqslant 1$ be an integer and $\mm =(m_1, \dots, m_n) \in \Z^n$ be a vector \textbf{coprime with} $d$. Let $\delta > 0$. For all $d$-admissible integer $q$, let $w_q$ be an element of order $d$ in $\zqzc$. For each $q$, we also choose subgroups of $\zqzc$: $H^{(1)}_q , \dots , H^{(n)}_q$. Then, provided for all $q$, there exists $i \in \{1, \dots, n\}$ such that
	$$|H_{q}^{(i)}| \geqslant q^\delta, $$
	the sets of $\varphi(d)$-tuples
	\begin{multline*} \bigg\{  \left(\frac{a_1 (w_q^{m_1})^0 + \dots + a_n (w_q^{m_n})^0}{q}, \dots, \frac{a_1 (w_q^{m_1})^{\varphi(d)-1} + \dots + a_n (w_q^{m_n})^{\varphi(d)-1}}{q} \right); \\  (a_1, \dots, a_n) \in H^{(1)}_q \times \dots \times H^{(n)}_q \bigg \}
	\end{multline*}
	become equidistributed modulo $1$ as $q$ goes to infinity among the $d$-admissible integers.
\end{prop}
\begin{proof}
	See section \ref{secpreuveequimyersoncase2}.
\end{proof}
Corollary \ref{coromyerson}, which generalizes the classical Myerson's lemma stated in the introduction, is obtained as the special case where $\mm = (1)$.\\

The Laurent polynomials $g_d$ introduced in Definition \ref{defgd} will carry the equidistribution result of Proposition \ref{equimyersoncase2} to the equidistribution theorem for exponential sums we are aiming at. This is the content of the following proposition.

\begin{prop} \label{reducmod1}
	Proposition $\ref{equimyersoncase2}$ implies case $\mathrm{(b)}$ of Theorem $\mathrm{\ref{thprincipal}}$.
\end{prop}

\begin{proof}
We still denote by $Y_{\mm,q} := H_q^{(1)} \times \dots \times H_q^{(n)}$ and by $\theta_{\mm, q}$ the map defined in \eqref{defthetamq}. We assume that the subgroups $H_q^{(i)}$ satisfy the growth condition \eqref{growth2} instead of \eqref{growth1}. Thanks to the proof of Proposition \ref{reducmod1case1} we have that

	$$\theta_{\mm,q}\left( a_1, \dots, a_n\right) = \sum ^{d-1}_{k=0} \prod_{i=1}^{n} \prod_{j=0}^{\varphi(d_i)-1} e\left(\frac{a_i (w_q^{m_i})^j}{q}\right)^{c_{j,k}^{(i)}}
	$$
	for all $(a_1, \dots, a_n) \in Y_{ \mm,q}$, where the integers $c_{j,k}^{(i)}$ are defined as in Definition \ref{deffd}. However, in the case where $\mm$ is coprime with $d$, all the $d_i$ are equal to $d$. This implies that for all $i$, the list of integers $(c_{j,k}^{(i)})_{0 \leqslant j <\varphi(d_i), \ 0 \leqslant k < d}$ is always equal to the list $(c_{j,k})_{0 \leqslant j <\varphi(d), \ 0 \leqslant k < d}$ of Definition \ref{defgd}.
Therefore,
	
	\begin{align*}\theta_{\mm,q}\left( a_1, \dots, a_n\right)= \sum ^{d-1}_{k=0} \prod_{j=0}^{\varphi(d)-1} e\left(\frac{a_1 (w_q^{m_1})^j + \dots + a_n (w_q^{m_n})^j}{q}\right)^{c_{j,k}}.
	\end{align*}
	
So if we define for all $j \in \{ 0, \dots, \varphi(d)-1 \}$, 
	
	\begin{equation} \label{introzj}
		z_j = z_j(a_1, \dots, a_n, q, j) :=  e\left(\frac{a_1 (w_q^{m_1})^j + \dots + a_n (w_q^{m_n})^j}{q}\right),
	\end{equation}
	
	we have \begin{equation} \label{gdzj}
	\theta_{\mm,q} \left( a_1, \dots, a_n\right) = g_d(z_0, \dots, z_{\varphi(d)-1}),
	\end{equation}
	with the Laurent polynomial $g_d$ defined at Definition \ref{defgd} and the $z_j$ being elements of $\T$.  This shows that $\theta_{\mm,q} \left( a_1, \dots, a_n\right) $ belongs to the image of $g_d$. Finally, the argument to prove the equidistribution with respect to the appropriate pushforward measure is the same as in the proof of Proposition \ref{reducmod1case1}, with Proposition \ref{equimyersoncase2} playing the role of Proposition \ref{equimyersoncase1}.
\end{proof}	

In the next section, we prove the key exponential sum estimate which will allow us to prove Proposition \ref{equimyersoncase1} and Proposition \ref{equimyersoncase2} using Weyl's equidistribution criterion.

\section{Improved versions of Myerson's lemma} \label{secmyerson}
In the previous section, we proved that Theorem \ref{thprincipal} (b) reduces to a statement on the equidistribution modulo $1$ of some specific sets of $\varphi(d)$-tuples. For instance when $\mm = (1)$, Proposition \ref{equimyersoncase2} reduces to the study of the distribution modulo $1$ of sets of the form
\begin{equation} \label{formeavecHq}
\left\{\frac{a}{q} \left(1, w_q, \dots , w_q^{\varphi(d)-1}\right); \ a \in H_q\right\},
\end{equation}
where $H_q$ is a sufficiently large subgroup of $\zqzc$ and $w_q$ is an element of order $d$ in $\zqzc$ (see Corollary \ref{coromyerson}). This can be seen as a generalization of \cite[Lemma 6.2]{periods}, which concerns sets of the form

\begin{equation} \label{formeaveczqz}
\left\{\frac{a}{q} \left(1, w_q, \dots , w_q^{\varphi(d)-1}\right); \ a \in \zqz\right\} \cdot
\end{equation}

If one wants to go from the equidistribution modulo $1$ of sets of the form \eqref{formeaveczqz} to that of sets of the form \eqref{formeavecHq}, this can become quite technical when the subgroup $H_q$ is small. More precisely, the equidistribution modulo $1$ of sets of the form \eqref{formeavecHq} can be proved by elementary means as long as the subgroups $H_q$ satisfy the growth condition (see \cite{mathese})
\begin{equation} \label{racinedeqsurhq}
\frac{\sqrt q}{|H_q|} \tend{q \to \infty} 0.
\end{equation}

However, as explained in \cite{kurlberg}, improving upon the range \enquote{$|H_q|$ grows faster than $\sqrt q$} requires more difficult techniques, and the best known bounds come from deep results from additive combinatorics. In this section, we will prove the key exponential sum estimate which allows us to prove the equidistribution modulo $1$ of the sets $\eqref{formeavecHq}$ for very small subgroups $H_q$. Precisely, it will give us the fact that the growth condition \eqref{racinedeqsurhq} can be replaced by
$$|H_q| \geqslant q^{\delta}$$
for any fixed $\delta >0$, which represents a huge improvement. This relies on a theorem due to Bourgain, that we state in the next subsection.

\subsection{Bourgain's theorem}
\label{boundsmult}
In a series of articles, Bourgain, Chang, Glibichuk and Konyagin proved very strong estimates on sums of additive characters modulo $q$ over subgroups of $\zqzc$, for different forms of factorization of $q$. The case where $q$ is prime is proved in \cite{BGK}, while the case of prime powers with bounded exponent is settled in \cite{bourgainchang}. This series of works culminated with the following theorem, which treats the general case, and includes in particular the case of small primes raised to very high powers. Notice that since we are only dealing with moduli which are prime powers, the proof of the result we really use is contained in Part I of \cite{bourgainarbitrary}.

\begin{theobis}[Bourgain] \label{bourgain}
	For any $\delta > 0$, there exists $\varepsilon = \varepsilon (\delta) > 0$ such that for any integer $q \geqslant 2$, and any subgroup $H$ of $\zqzc$ such that $|H| \geqslant q^\delta$, 
	\begin{equation}
	\max_{a\in \zqzc} \left| \sum_{x \in H}^{} e \left( \frac{ax}{q}\right)\right| \leqslant C \frac{|H|}{q^\varepsilon}
	\end{equation}
	where $C$ is a constant depending at most on $\delta$.
\end{theobis}

\begin{proof}
	See \cite[Theorem]{bourgainarbitrary}.
\end{proof}

\subsection{The crucial control of the $p$-adic valuation}
In order to apply the previous theorem in our specific context, the following proposition plays an important role.

\begin{prop} \label{lembezout}
	Let $d \geqslant 1$ and let $f\in \Z[X] \setminus \{0\}$ be a polynomial of degree strictly less than $\varphi(d)$. For all $d$-admissible integers $q$, we choose an element $w_q$ of order $d$ in $\zqzc$, and an arbitrary lift $\tilde w_q$ in $\Z$. Then there exist two constants $C_f, n_f \geqslant 1$ such that for all $q= p^\alpha \in \mathcal A_d$ such that $q> n_f$,
	\begin{enumerate}[label=$\mathrm{(\alph*)}$]
		\item  $ q$ does not divide $f(\tilde w_q). $
		\item $p^{v_p(f(\tilde w_q))} \leqslant C_f.$
	\end{enumerate}
	
\end{prop}

\begin{proof}	As in the proof of \cite[Lemma 6.2]{periods}, we use the fact that there exist two polynomials $a,b \in \Z[X]$ and an integer $n \geqslant 1$ such that
	\begin{equation}\label{polybezoutgauss} a(X)\phi_d(X) + b(X) f(X)  = n, 
	\end{equation}
since $f$ and $\phi_d$ are coprime in the euclidean domain $\Q[X]$. Now, let $q= p^\alpha$ be a $d$-admissible integer. Reducing equation \eqref{polybezoutgauss} modulo $q$ and evaluating it at $w_q$ leads to
	$$ \overline a ( w_q) \overline \phi_d( w_q) + \overline b ( w_q) \overline f ( w_q) \equiv n \mod p^\alpha $$
hence 
	\begin{equation} \label{bfn}
	\overline b ( w_q) \overline f ( w_q) \equiv n \mod p^\alpha 
	\end{equation}
by Lemma \ref{affirmation}. Now, if $q = p^\alpha > n$, then $n$ is non-zero modulo $q$, hence $\overline f ( w_q) \not \equiv 0  \mod q$. This precisely means that $q$ does not divide $f(\tilde w_q)$. This shows that $n_f:= n$ is a suitable constant for assertion (a). This part of the lemma actually completes the proof of Lemma \ref{myerson}. \\
	
	Another way of phrasing what we just proved is that as soon as $q > n$, the $p$-adic valuation of $f(\tilde w_q)$ is strictly less than $\alpha$. Let us denote by $\gamma <  \alpha$ the $p$-adic valuation of $f(\tilde w_q)$. Then, if we reduce the congruence \eqref{bfn} modulo $p^\gamma$, we get $n \equiv 0 \mod p^\gamma$. Thus,
	$$\gamma = v_p(f(\tilde w_q)) \leqslant v_p(n),$$
	hence
	$$p^{v_p(f(\tilde w_q))} \leqslant p^{ v_p(n)} \leqslant n.$$
	Therefore, we proved that with the choice $C_f := n$, assertion (b) holds.
\end{proof}

Now, let us use this proposition, together with Theorem \ref{bourgain}, to deduce the key exponential sum estimate which will be used in the proofs of Proposition \ref{equimyersoncase1} and Proposition \ref{equimyersoncase2}.

\subsection{Proof of the key exponential sum estimate: Proposition \ref{myersongen}}

	Let $q = p^ \alpha$ be a $d$-admissible integer, and let $H_q$ and $w_q$ be as in the statement. Let $ \tilde w_q$ be any lift in $ \Z$ of the class $w_q$. Assume further that $q > n_f$ for any constant $n_f$ as in Proposition \ref{lembezout}.\\
	One cannot directly apply Bourgain's theorem to the sum on the left-hand side of \eqref{sommefwq} because $f(w_q)$ could be non-invertible modulo $q$. In order to reduce to a situation where Bourgain's theorem applies, let us introduce the notation $ \beta_q$ for the $p$-adic valuation of $f(\tilde w_q)$, and write $f(\tilde w_q) := p^{\beta_q} r_q$. By Proposition \ref{lembezout} (a), $\beta_q <  \alpha$. Then we have
	\begin{equation} \label{formerq}
	\sum_{a \in H_{q}}^{} e \left( \frac{af(w_{q})}{q}\right) = \sum_{a \in H_{q}}^{} e \left( \frac{ar_q}{p^{\alpha -  \beta_q}}\right).
	\end{equation}
	Now, each of the terms $e \left( \frac{ar_q}{p^{\alpha -  \beta_q}}\right)$ only depends on the class of $a$ modulo $p^{\alpha -  \beta_q}$. Let us denote by $q' := p^{\alpha -  \beta_q}$ and by $\pi$ the group homomorphism $(\Z/ q   \Z)^\times \to (\Z / q' \Z)^\times$ induced by the reduction modulo $q'$. The latter induces a group homomorphism $\tilde \pi \colon H_q \to \pi(H_q) =: H'_q$. We denote by $k := |\ker \tilde \pi |$. Then we have the following equality:
	\begin{equation*} 
	\sum_{a \in H_{q}}^{} e \left( \frac{ar_q}{p^{\alpha -  \beta_q}}\right) = k\sum_{a \in H'_{q}}^{} e \left( \frac{ar_q}{q'}\right).
	\end{equation*}
	Indeed, any element of $H'_q$ has exactly $k$ preimages in $H_q$ under the reduction modulo $q'$.\\
	Since $\ker \tilde \pi \subseteq \ker \pi$, we have that $k \leqslant | \ker \pi | = p^{\beta_q}$. Therefore,
	\begin{equation} \label{pbetaq}
	\left|\sum_{a \in H_{q}}^{} e \left( \frac{ar_q}{p^{\alpha -  \beta_q}}\right)\right| \leqslant p^{\beta_q} \left| \sum_{a \in H'_{q}}^{} e \left( \frac{ar_q}{q'}\right) \right|.
	\end{equation}
	In order to apply Theorem  \ref{bourgain} to the sum on the right-hand side, we first need to check that the subgroup $H'_q$ of $\left( \Z / q' \Z \right)^\times$ is large in the following sense: $|H'_q| \geqslant (q')^{\delta'}$ for some $\delta' > 0$. Using the fact that $|H_q| \geqslant q^\delta$, we have
	$$ |H'_q| = \frac{|H_q|}{k} \geqslant \frac{|H_q|}{p^{\beta_q}} \geqslant \frac{q^\delta}{p^{\beta_q}} = \frac{(q')^\delta}{\left(p^{\beta_q}\right)^{1-\delta}} \geqslant \frac{(q')^\delta}{C_f^{1-\delta}},$$
	where the last lower bound comes from the inequality $p^{\beta_q} \leqslant C_f$ given by Proposition \ref{lembezout} (b). This inequality also ensures that $q'$ tends to infinity as $q$ goes to infinity, so that $\frac{(q')^{\frac{\delta}{2}}}{C_f^{1- \delta}}$ eventually becomes greater than $1$ as $q$ becomes large. Therefore,
	\begin{equation} \label{deltasurdeux}
	|H'_q| \geqslant (q')^{\frac{\delta}{2}}
	\end{equation}
	for all $q$ large enough. Thanks to \eqref{deltasurdeux}, Theorem \ref{bourgain} applies to the sum on the right-hand side of \eqref{pbetaq}, because we also have that $r_q$ is invertible modulo $q'$. So there exists a constant $\varepsilon = \varepsilon(\delta /2) > 0$ and a constant $C$ depending at most on $\delta$ such that
	$$\left| \sum_{a \in H'_{q}}^{} e \left( \frac{ar_q}{q'}\right) \right| \leqslant C \frac{|H'_q|}{(q')^\varepsilon} \cdot$$
	Thanks to \eqref{formerq} and \eqref{pbetaq}, this implies the following upper bound:
	$$\left|\sum_{a \in H_{q}}^{} e \left( \frac{af(w_{q})}{q}\right)\right| \leqslant  C \frac{p^{\beta_q}|H'_q|}{(q')^\varepsilon} \leqslant C \frac{p^{\beta_q}|H_q|}{\left(p^{\alpha - \beta_q}\right)^\varepsilon} =  C \frac{|H_q|p^{\beta_q (1+\varepsilon)}}{q^\varepsilon} \leqslant CC_f^{1+\varepsilon} \frac{|H_q|}{q^\varepsilon} \ll_{f, \delta} \frac{|H_q|}{q^\varepsilon}$$
	which concludes the proof. \qed

\section{Proofs of the propositions on equidistribution modulo 1}	\label{seclemmamod1}

\subsection{Proof of Proposition \ref{equimyersoncase1} (Equidistribution modulo $1$ case (a))} \label{secpreuveequimyersoncase1}

	We are interested in the equidistribution modulo $1$ of the following sets of $(\varphi(d_1) + \dots + \varphi(d_n))$-tuples (with the slight abuse of notation underlined at remark \ref{abusdenot}):
	
	$$\bigg \{  \overbrace{\left( \left(\frac{a_1 ( w_q^{m_1})^j}{q}\right)_{0 \leqslant j <  \varphi(d_1)} , \dots, \left(\frac{a_n (w_q^{m_n})^j}{q}\right)_{0 \leqslant j < \varphi(d_n)} \right)}^{=: \xx(a_1, \dots, a_n,q)}; \quad (a_1, \dots, a_n) \in H^{(1)}_q \times \dots \times H^{(n)}_q \bigg \},
	$$	 
	where the $H_q^{(i)}$ are subgroups of $\zqzc$ satisfying the following growth condition:
	$$\forall q \in \mathcal A_d, \forall i \in \{1, \dots, n\}, \  |H_q^{(i)}| \geqslant q^\delta$$ for some $\delta > 0$.
	By Weyl's criterion (see \cite[Theorem 6.2]{kuipers}), these sets become equidistributed modulo 1 if and only if for any $\yy = \left(y_0, \dots, y_{\varphi(d_1)+ \dots + \varphi(d_n)-1}\right) \in \Z^{\varphi(d_1)+ \dots + \varphi(d_n)} \setminus \{0\}$, we have the following convergence towards zero:
	
	\begin{equation} \label{retendvers0} 	\frac{1}{\prod_{i= 1}^{n} |H_q^{(i)}|} \times \  \sum_{\substack{a_1 \in H^{(1)}_q \\ \svdots \\ a_n \in H_q^{(n)}}} e\left(\xx(a_1,\dots,a_n,q)\cdot \yy\right) \tend{\substack{q \to \infty \\  q \in \mathcal A_d}} 0.
	\end{equation}

	Let us denote by $\yy_1$ the vector extracted from $\yy$ by taking the first $\varphi(d_1)$ entries, $\yy_2$ the vector formed by the next $\varphi(d_2)$ entries and so on:
	$$\yy_1 = (y_0, \dots, y_{\varphi(d_1)-1}), \quad \yy_2 = (y_{\varphi(d_1)} ,  \dots, y_{\varphi(d_1) + \varphi(d_2) - 1}) \quad \yy_3 = \cdots $$
	so that $\yy = (\yy_1, \dots, \yy_n)$. We also introduce the following notations to decompose the vector $\xx(a_1, \dots, a_n,q)$ in a parallel way:
	
	$$\xx_1(a_1,q) := \left(\frac{a_1 ( w_q^{m_1})^j}{q}\right)_{0 \leqslant j <  \varphi(d_1)}, \dots,\  \xx_n(a_n,q):= \left(\frac{a_n ( w_q^{m_n})^j}{q}\right)_{0 \leqslant j < \varphi(d_n)}$$
	Then we have
	
	\begin{equation} \label{produit}
	\frac{1}{\prod_{i= 1}^{n} |H_q^{(i)}|} \times \  \sum_{\substack{a_1 \in H^{(1)}_q \\ \svdots \\ a_n \in H_q^{(n)}}} e\left(\xx(a_1,\dots,a_n,q)\cdot \yy\right) =\prod_{i=1}^{n} \left[\frac{1}{ |H_q^{(i)}|} \sum_{a_i \in H_q^{(i)}}^{} e(\xx_i(a_i,q) \cdot \yy_i)\right].
	\end{equation}
	
	Now, since $\yy \neq 0$, there exists at least one index $i \in \{1, \dots, n\}$ such that $\yy_{i} \neq 0$. For such an $i$, the factor
	
	\begin{equation} \label{termei}
\frac{1}{ |H_q^{(i)}|} \sum_{a_i \in H_q^{(i)}}^{} e(\xx_i(a_i,q) \cdot \yy_i)
	\end{equation}
	
	tends to $0$ as $q$ goes to infinity among the $d$-admissible integers thanks to Proposition \ref{myersongen}. Indeed, we have
	
$$\frac{1}{ |H_q^{(i)}|} \sum_{a_i \in H_q^{(i)}}^{} e(\xx_i(a_i,q) \cdot \yy_i)=	\frac{1}{|H_q^{(i)}|} \sum_{a \in H_q^{(i)}}^{} e\left( \frac{af_i(w_q^{m_i})}{q}  \right),$$
where $f_i$ is the polynomial associated with $\yy_i =\left(y_{\varphi(d_1) + \dots + \varphi(d_{i-1})}, \dots, y_{\varphi(d_1) + \dots + \varphi(d_{i-1}) + \varphi(d_i)-1} \right)$ as follows: $f_i = y_{\varphi(d_1) + \dots + \varphi(d_{i-1})} + y_{\varphi(d_1) + \dots + \varphi(d_{i-1}) +1}  X + \dots + y_{\varphi(d_1) + \dots +\varphi(d_{i-1}) + \varphi(d_i)-1} X^{\varphi(d_i)-1} $. This is a non-zero polynomial with integer coefficients and with degree strictly less than $\varphi(d_i)$, and $w_q^{m_i}$ is an element of order $d_i$ in $\zqzc$. Thus, we can apply Proposition \ref{myersongen} which states that there exists a rank $N_{f_i}$ such that for all $q > N_{f_i}$ such that $q$ is $d$-admissible,
$$ \left| \sum_{a \in H_q^{(i)}}^{} e\left( \frac{af_i( w_q^{m_i})}{q}  \right)\right| \ll_{\delta, f_i} \frac{|H_q^{(i)}|}{q^{\varepsilon(\delta)}}, $$ 
and this suffices to prove the convergence of \eqref{termei} towards zero. As all the other factors of \eqref{produit} have absolute value bounded above by $1$, the whole product converges to zero, and this concludes the proof. \qed

\begin{rem} 
	The proof shows why it is important to ask for the growth condition \eqref{growth1} instead of \eqref{growth2} (that will be used in the next proof). Indeed, let us fix an index $j \in \{1, \dots, n\}$. Then if we take $\yy = (\yy_1, \dots, \yy_n) \in \Z^{\varphi(d_1)+ \dots + \varphi(d_n)} \setminus \{0\}$ defined by $ \yy_i = (0, \dots, 0) \in \Z^{\varphi(d_i)}$ for all $i \neq j$ and $\yy_j = (1, \dots, 1) \in \Z^{\varphi(d_j)}$, then the absolute value of the product \eqref{produit} is equal to the absolute value of the factor corresponding to the index $j$, since all the other factors are equal to $1$. Therefore, to prove the convergence towards zero in Weyl's criterion for this specific vector $\yy$, we have no other choice than proving that the factor corresponding to the index $j$ tends to $0$. In order to achieve that, we really need to be able to apply Proposition \ref{myersongen} to this factor, hence we really need to require $|H_q^{(j)}| \geqslant q^\delta$. As $j$ was arbitrary, this shows that the growth condition needs to be satisfied for all $j \in \{1, \dots, n\}$.
\end{rem}

\subsection{Proof of Proposition \ref{equimyersoncase2} (Equidistribution modulo $1$ case (b))} \label{secpreuveequimyersoncase2}
 We are interested in the equidistribution modulo $1$ of the following sets of $\varphi(d)$-tuples: 

\begin{multline*} \bigg\{  \overbrace{\left(\frac{a_1 ( w_q^{m_1})^0 + \dots + a_n ( w_q^{m_n})^0}{q}, \dots, \frac{a_1 ( w_q^{m_1})^{\varphi(d)-1} + \dots + a_n ( w_q^{m_n})^{\varphi(d)-1}}{q} \right)}^{=: \xx(a_1,\dots, a_n, q)}; \\  (a_1, \dots, a_n) \in H^{(1)}_q \times \dots \times H^{(n)}_q \bigg \},
\end{multline*}
where the $H_q^{(i)}$ are subgroups of $\zqzc$ satisfying the growth condition
\begin{equation} \label{growth2reprise}
	\forall q \in \mathcal A_d, \exists i \in \{1, \dots, n\}, \quad |H_q^{(i)}| \geqslant q^{\delta}.
\end{equation}

By Weyl's criterion, these sets become equidistributed if and only if for any $\yy := \left(y_0, \dots, y_{\varphi(d) -1} \right) \in \Z^{\varphi(d)} \setminus \{0\}$ we have the following convergence towards zero:

\begin{equation} \label{tendvers0} \frac{1}{\prod_{i= 1}^{n} |H_q^{(i)}|}\times \  \left( \sum_{(a_1, \dots, a_n) \in H^{(1)}_q \times \dots \times H^{(n)}_q}^{} e\left(\xx(a_1,\dots,a_n,q)\cdot \yy\right) \right) \tend{\substack{q \to \infty \\  q \in \mathcal A_d}} 0
\end{equation}

But the left-hand side can be rewritten as:

\begin{equation} \label{cadepasse} 
	 \prod_{i= 1}^{n} \left[ \frac{1}{|H_q^{(i)}|} \sum_{a_i \in H_q^{(i)}}^{} e\left( \frac{a_if( w_q^{m_i})}{q} \right) \right]
 \end{equation}
where $f$ is the polynomial $y_0 + y_1X + \dots + y_{\varphi(d)-1} X^{\varphi(d)-1}$.\\

Now, since $(m_i,d) = 1$, we have that for all $i \in \{1, \dots, n\}$ the element $w_q^{m_i}$ is still of order $d$ in $\zqzc$. Also, $f \in \Z[X] \setminus \{0\}$ and $\deg f < \varphi(d)$. Thus, we can use Proposition \ref{myersongen} to bound the modulus of the inner sum for any index $i$ such that $|H_q^{(i)} | \geqslant q^\delta$ (and there exists at least one such index $i$ thanks to the growth condition \eqref{growth2reprise}).  \\
We deduce that there exists an integer $N_f$ such that for all $q > N_f$ such that $q$ is $d$-admissible, we have:
	$$ \frac{1}{|H_q^{(i)}|} \left| \sum_{a_i \in H_q^{(i)}}^{} e\left( \frac{a_if( w_q^{m_i})}{q} \right) \right| \ll_{f, \delta} \frac{1}{q^\varepsilon} \quad \text{ where } \varepsilon = \varepsilon(\delta) > 0$$
	for at least one $i \in \{1, \dots, n\}$, which may vary with $q$. As the other factors in the product \eqref{cadepasse} have absolute value bounded above by one, we may conclude that the whole product has its modulus bounded above by $1/ q^\varepsilon$ (up to multiplicative constants), hence the conclusion. \\ \qed

\paragraph{Conclusion of the proof of Theorem \ref{thprincipal}.} In this section, we proved Proposition \ref{equimyersoncase1} and Proposition \ref{equimyersoncase2}, and each of them implies one of the cases of Theorem \ref{thprincipal} thanks to propositions \ref{reducmod1case1} and \ref{reducmod1}. This completes the proof of our main result.

\subsection{Remarks on Theorem \ref{thprincipal} and connexion with previous results}

\begin{rem} \label{fix}
There is one last refinement that we did not discuss so far but that can be observed through the proofs: it is the fact that one can fix some of the coefficients $a_i$ when one studies the equidistribution of the restricted sums
$$\sum_{\substack{x \in (\zqz)^\times \\ x^d = 1}} e\left(\frac{a_1 x^{m_1} + \dots + a_n x^{m_n}}{q}\right)$$
in the case where the $m_i$ are coprime with $d$. This remark is not included in Theorem \ref{thprincipal} (b) because the statement was already quite long without it. \\
Let $s \in \{1, \dots, n\}$ and let $\{i_1, \dots, i_s\} \subseteq \{1, \dots, n\}$. We fix $n-s$ integers $a_i$ for $i \in \{1, \dots, n\} \setminus \{i_1, \dots, i_s\}$. Then the sets of sums:

\begin{equation} \label{somefixed}	\left\{ \sum_{\substack{x \in (\zqz)^\times \\ x^d = 1}} e\left(\frac{a_1 x^{m_1} + \dots + a_n x^{m_n}}{q}\right); \ (a_{i_1}, \dots,a_{i_s}) \in H_q^{(i_1)} \times \dots \times H_q^{(i_s)} \right\} 
\end{equation}
(in other words, the ones where only some of the $a_i$ are free) become equidistributed in the image of $g_d$ with respect to the same measure as in Theorem \ref{thprincipal} (b) provided there exists $\delta > 0$ such that

\begin{equation} \label{growthi1is}
\forall q \in \mathcal A_d, \exists i \in  \{i_1, \dots, i_s \}, \quad  |H_{q}^{(i)}| \geqslant q^\delta
\end{equation}

Indeed, when one reduces this question to an equidistribution result modulo $1$ as in section \ref{secreduccase2}, one needs to prove that the sets

\begin{multline*} \bigg\{ \left(\frac{a_1 ( w_q^{m_1})^0 + \dots + a_n ( w_q^{m_n})^0}{q}, \dots, \frac{a_1 ( w_q^{m_1})^{\varphi(d)-1} + \dots + a_n ( w_q^{m_n})^{\varphi(d)-1}}{q} \right); \\  (a_{i_1}, \dots, a_{i_s}) \in H^{(i_1)}_q \times \dots \times H^{(i_s)}_q \bigg \}
\end{multline*}
become equidistributed modulo $1$. After applying Weyl's criterion, one gets a factor with complex absolute value equal to $1$, multiplied by the product

\begin{equation} \label{prodi1is} \prod_{i \in \{i_1, \dots,i_s\}}^{} \left[ \frac{1}{|H_q^{(i)}|} \sum_{a_i \in H_q^{(i)}}^{} e\left( \frac{a_if( w_q^{m_i})}{q}  \right) \right] 
\end{equation}

which tends to zero under condition \eqref{growthi1is}, thanks to Proposition \ref{myersongen}.
\end{rem}

\begin{rem}
If we take the subgroups $H_q^{(i)}$ as large as possible, Theorem \ref{thprincipal} gives an equidistribution result for the sets of sums
$$\left\{ \sum_{\substack{x \in (\zqz)^\times \\ x^d = 1}} e\left(\frac{a_1 x^{m_1} + \dots + a_n x^{m_n}}{q}\right); \ a_1, \dots, a_n \in \zqzc \right\}.$$
However, to compare our result to the ones already proved in \cite{periods, menagerie} and \cite{visual}, we would like to have the $a_i$ varying in $\zqz$ instead of $\zqzc$. This is the content of Proposition \ref{propavecmyersonfaible}. In order to obtain this proposition, the first step consists in following the strategy of section \ref{grandesectionreduc} to reduce to two statements on equidistribution modulo $1$, corresponding to the two cases (a) and (b) of the proposition. These two statements are the exact analogues of Proposition \ref{equimyersoncase1} and Proposition \ref{equimyersoncase2}, with the condition $(a_1, \dots, a_n) \in H_q^{(1)} \times \dots \times H_q^{(1)}$ replaced by $(a_1, \dots, a_n) \in \left(\zqz\right)^n$. They are proved using Weyl's criterion, but the convergence towards zero is easier to obtain, since Lemma \ref{myerson} replaces the use of Proposition \ref{myersongen}. For instance, Proposition \ref{propavecmyersonfaible} (b) reduces to the convergence towards zero of the following product

	\begin{equation} \label{prodBb}  \prod_{i \in \{i_1, \dots,i_s\}}^{} \left[ \frac{1}{q} \sum_{a_i \in \zqz}^{} e\left( \frac{a_if(w_q^{m_i})}{q}  \right) \right], \end{equation}
	where $f$ is a non-zero polynomial in $\Z[X]$ with $\deg f < \varphi(d)$. This product is to Proposition \ref{propavecmyersonfaible} (b) what \eqref{prodi1is} was to Theorem \ref{thprincipal} (b), in its extended form of the previous remark. Now, thanks to Lemma $\ref{myerson}$ each of the factors in \eqref{prodBb} is eventually equal to zero as $q \in \mathcal A_d$ goes to infinity, and this finishes the proof. Proposition \ref{propavecmyersonfaible} (a) can be proved by a similar adaptation of the proof of Theorem \ref{thprincipal} (a).

\end{rem}

This proposition generalizes the previous results obtained in \cite{periods, menagerie} and \cite{visual}. In these articles, only sums of the type
$$\sum_{x^d = 1}^{} e\left( \frac{ax}{q}\right) \quad \text{ or } \quad	\sum_{x^d = 1}^{} e\left( \frac{ax+b x^{-1}}{q}\right)$$
were considered. In Proposition \ref{propavecmyersonfaible} (b), we prove that the equidistribution theorems obtained in \emph{loc. cit.} extend to families of sums with a more general numerator inside the exponentials, namely Laurent polynomials of the form $a_1 x^{m_{1}}+ \dots + a_n x^{m_n}$, as soon as the exponents $m_i$ are coprime with $d$. 
Moreover, it also explains that one can fix some of the coefficients $a_i$. As long as one of them varies in all $\zqz$, the equidistribution result will hold. In particular, the equidistribution results obtained in \cite{visual} for sets of Kloosterman sums
$$\mathcal K_q(-,-,d) := \left\{ \K_q(a,b, d):= \sum_{x^d = 1}^{} e\left( \frac{ax+b x^{-1}}{q}\right); \ a,b \in (\zqz)^2 \right\}$$
also hold if one fixes an integer $a_0$, and considers the sets
$$\mathcal K_q(a_0,-,d) := \left\{ \K_q(a_0,b, d); \ b \in \zqz \right\}.$$
Finally, case (a) treats the case where some of the $m_i$ may share prime factors with $d$, and provides the appropriate Laurent polynomial whose image determines the region of equidistribution. 	

\section{Examples of application}
\label{notables} 

\subsection{Laurent polynomials and hypocycloids}

So far, the region of the complex plane in which the sums become equidistributed has only been described as the image of a torus via a Laurent polynomial. However, in some cases, one can give a more explicit description of this image, as it was observed in \cite{menagerie,visual}. In this section, we review some cases where the regions of equidistribution can be described in geometric terms, without referring to the Laurent polynomials $g_d$.

The main ingredient is the following lemma, which states that the region $\H_d$ from Definition \ref{defHd} is precisely the image of a torus via a specific Laurent polynomial.

\begin{lem} \label{image}
	Let $d \geqslant 2$. The image of the map: $$\begin{array}{ccccc}
		f & : & \T^{d-1} &  \to & \C \\
		&& (z_1, \dots, z_{d-1})&  \mapsto & z_1 + \dots+z_{d-1}  + \frac{1}{z_1 \dots z_{d-1}}
	\end{array} $$
	is the region $\H_d$ from Definition $\ref{defHd}$.
\end{lem}
	\begin{proof}
	See \cite[Theorem 3.2.3]{cooper} or \cite[section 3]{kaiser}. Note that this is equivalent to asking the question: \enquote{which complex numbers arise as the trace of a matrix in $\SU(d)$?}
\end{proof}

The above Laurent polynomial arises in our equidistribution results in the following way: when $d$ is a prime number, the $d^{\text{th}}$ cyclotomic polynomial $\phi_d$ is equal to $X^{d-1} + \dots + X +1$, and this explicit formula allows one to compute the coefficients $c_{j,k}$ which appear in the definition of $g_d$ (Definition $\ref{defgd}$). This exactly gives the Laurent polynomial of the previous lemma, thus allowing us to have a more concrete description of the region of equidistribution in our results. Precisely, this gives:
\begin{prop}[{\cite[Proposition 1]{menagerie}}] \label{interpret}
	Let $d$ be a prime number. The polynomial $g_d$ from Definition $\ref{defgd}$ is given by $$ \begin{array}{ccccc}
		g_d&:& \T^{\varphi(d)} = \T^{d-1}& \to & \C \\
		&& (z_1,\dots, z_{d-1}) & \mapsto & z_{1}+\ldots +z_{d-1}+\dfrac{1}{z_{1}z_{2}\ldots z_{d-1}}
	\end{array}$$
	and the image of $\T^{d-1}$ via $g_d$ is the region $\H_d$ from Definition $\ref{defHd}$. In particular, the region of the complex plane in which the sums restricted to the subgroup of order $d$ become equidistributed in Theorem $\mathrm{\ref{thprincipal}~(b)}$ and Proposition $\mathrm{\ref{propavecmyersonfaible}~(b)}$ is $\H_d$.
\end{prop}

This proposition relies mostly on the fact that when $d$ is a prime number, we have an explicit formula for the $d^{\text{th}}$ cyclotomic polynomial. As there is also an explicit formula for the $d^{\text{th}}$ cyclotomic polynomial when $d= r^b$ is a prime power, namely

$$\phi _{r^b}\left( X\right) =\sum ^{r-1}_{j=0}X^{jr^{b-1}} \left( = \phi_r\left(X^{r^{b-1}}\right)\right),$$
it is not surprising that our understanding of the image of $g_d$ can also be improved in that case. In fact, the explicit formula above leads to the following proposition.

\begin{prop}[{\cite[Corollary 1]{menagerie}}] \label{propgrb}
	Let $d:= r^b$ be a power of a prime number $r$. The polynomial $g_d$ from Definition $\ref{defgd}$ is given by $$ \begin{array}{ccccc}
		g_d&:& \T^{\varphi(d)} = \T^{(r-1)r^{b-1}}& \to & \C \\
		&& (z_1,\dots, z_{(r-1)r^{b-1}}) & \mapsto & \displaystyle \sum ^{(r-1)r^{b-1} }_{j=1}z_{j}+\sum ^{r^{b- 1}}_{m=1}\prod ^{r-2}_{\ell=0}z_{m+\ell r^{b-1}}^{-1}
	\end{array}$$
	and the image of $\T^{\varphi(d)}$ via $g_d$ is the Minkowski sum
	$$\sum_{j=1}^{r^{b-1}} \H_r := \left\{ \xi_1 + \cdots + \xi_{r^{b-1}}; \ \xi_1, \dots , \xi_{r^{b-1}} \in \H_r \right\}.$$ 
	In particular, the region of the complex plane in which the sums restricted to the subgroup of order $d$ become equidistributed in Theorem $\mathrm{\ref{thprincipal}~(b)}$ and Proposition $\mathrm{\ref{propavecmyersonfaible}~(b)}$ is $\displaystyle \sum_{j=1}^{r^{b-1}} \H_r.$
\end{prop} 
\begin{exemple} \label{g9}
		For instance, as it is done in \cite[Theorem 10]{visual}, for $r=3$ and $b=2$ we have:
	$$g_9(z_1, \dots, z_6) =\underbrace{ z_{1}+z_{4}+\dfrac{1}{z_{1}z_{4}}}_{\in \H_3}+ \underbrace{z_{2}+z_{5}+\dfrac{1}{z_{2}z_{5}}}_{\in \H_3}+\underbrace{z_{3}+z_{6}+\dfrac{1}{z_{3}z_{6}}}_{\in \H_3} \cdot$$
	A drawing of the region of the complex plane $\H_3 + \H_3 + \H_3$ can be found in \cite[Figure 11]{menagerie}.
\end{exemple}

\subsection{Illustration of Theorem \ref{thprincipal}} \label{illusthprincipal}
We fix $d$ to be equal to $5$, and we consider the following sets of Kloosterman sums restricted to the subgroup of order $5$:
\begin{equation} \label{varieSG}
	\left\{ \sum_{\substack{x \in (\zqz)^\times \\ x^5 = 1}} e\left(\frac{a x + b x^{-1}}{q}\right); \ (a,b) \in H_q^{(1)} \times H_q^{(2)}  \right\}
\end{equation}
for increasing $5$-admissible values of $q$, and where $H_q^{(1)}$ and $H_q^{(2)}$ denote subgroups of $\zqzc$. Then Theorem \ref{thprincipal} (b), combined with the geometric interpretation of Proposition \ref{interpret}, states that these sets become equidistributed in the region delimited by a $5$-cusp hypocycloid, with respect to some measure (obtained as the pushforward measure, under the Laurent polynomial $g_5$, of the Haar measure on $\T^4$) \textit{provided the subgroups $H_q^{(1)}$ and $H_q^{(2)}$ of $\zqzc$ satisfy the following growth condition}:\\
there exists $\delta > 0$ such that:
$$\forall q \in \mathcal A_d, \quad |H_q^{(1)} |\geqslant q^\delta \text{ or } |H_q^{(2)}| \geqslant q^\delta .$$
In the following figure, $H_q^{(2)}$ is always chosen to be the trivial multiplicative subgroup, and $|H_q^{(1)}| \geqslant q^{1/2}$.

\begin{figure}[H]
	\captionsetup[subfigure]{justification=centering}
	\centering
	\begin{subfigure}[b]{0.32\textwidth}
		\centering
		\includegraphics[width=\textwidth]{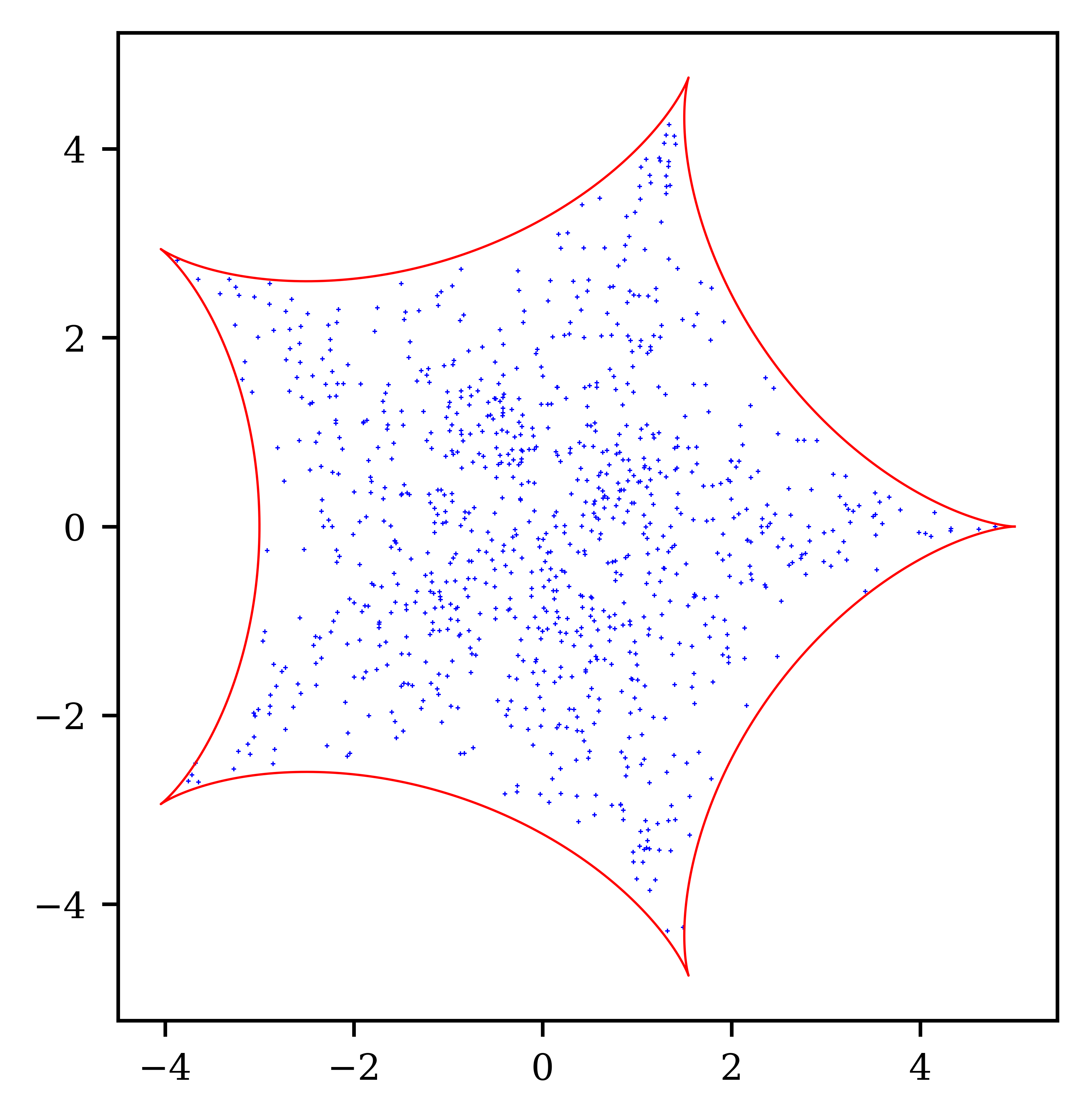}
		\caption{$q = 1901$ \\
			$|H_q^{(1)}| = 950 $\\
			$|H_q^{(2)}|= 1$	}
	\end{subfigure}
	\hfill
	\begin{subfigure}[b]{0.32\textwidth}
		\centering
		\includegraphics[width=\textwidth]{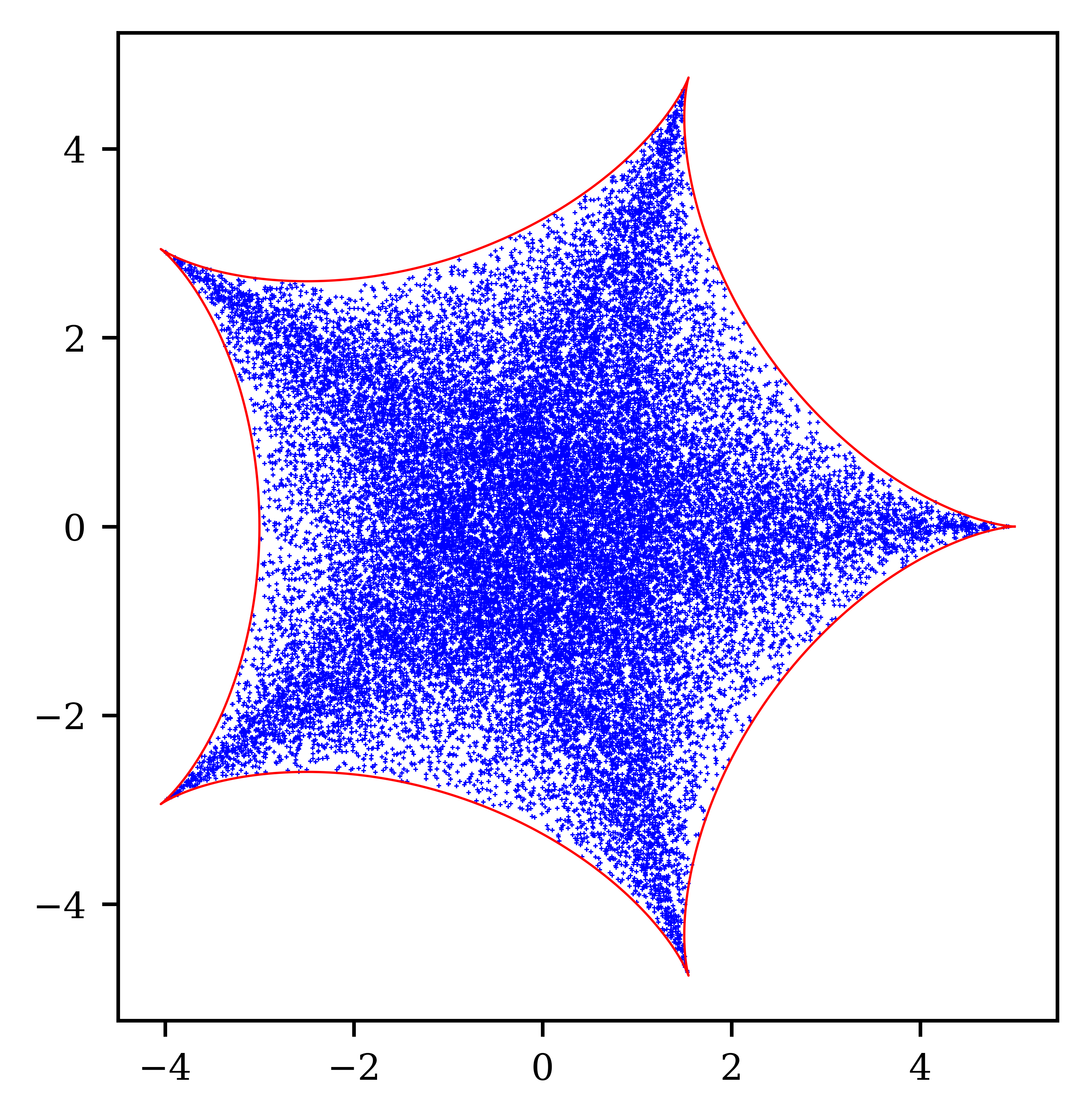}
		\caption{$q = 421^2$ \\
			$|H_q^{(1)}| = 29470 $\\
			$|H_q^{(2)}|= 1$	}
	\end{subfigure}
	\hfill
	\begin{subfigure}[b]{0.32\textwidth}
		\centering
		\includegraphics[width=\textwidth]{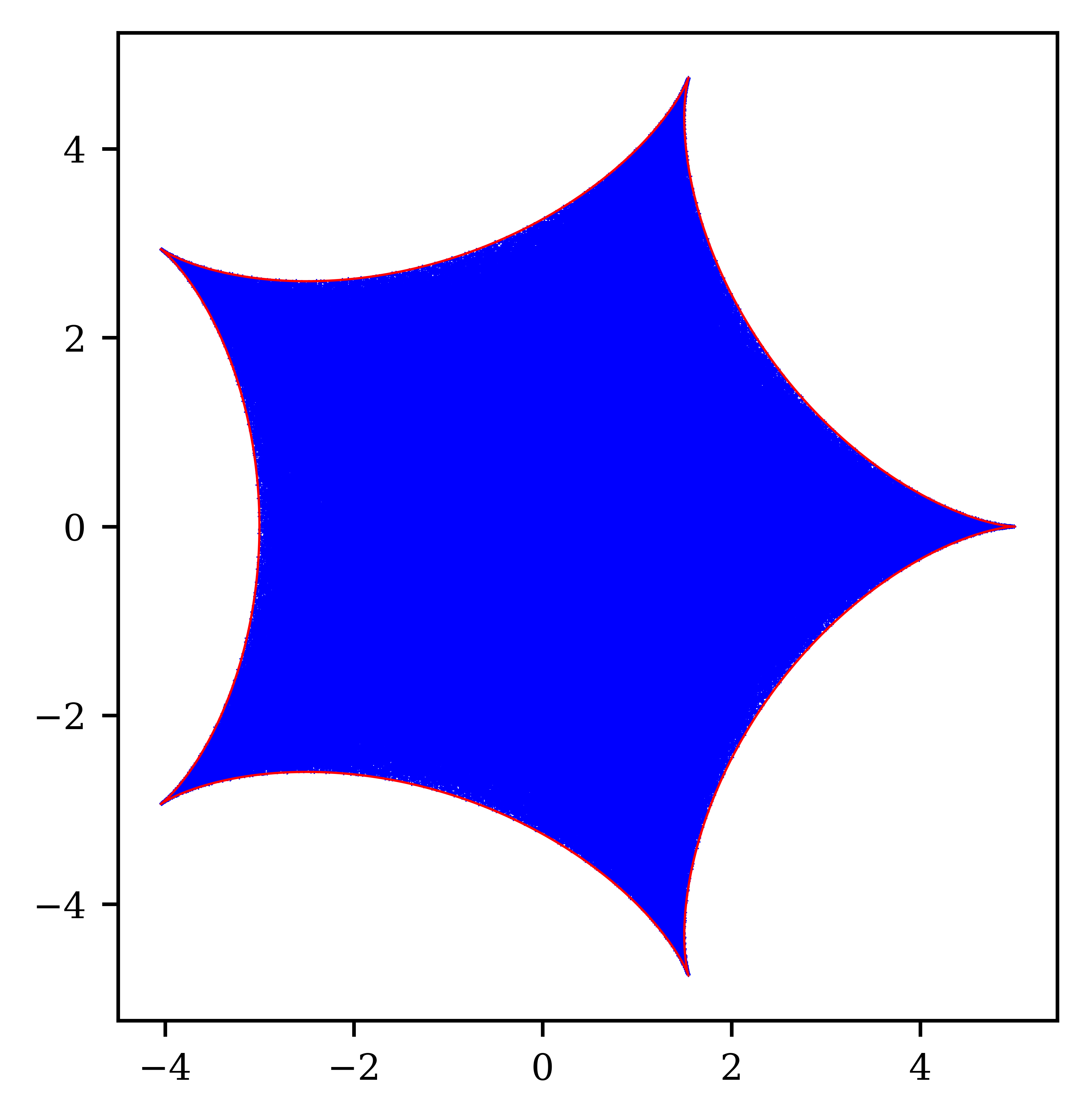}
		\caption{$q = 971^2$ \\
			$|H_q^{(1)}| = 941870 $\\
			$|H_q^{(2)}|= 1$	}
	\end{subfigure}
	\caption{The sets of the form \eqref{varieSG} for three $5$-admissible integers $q$ and for the indicated choice of subgroups $H_q^{(1)}$, $H_q^{(2)}.$}
	
\end{figure}

\subsection{Illustrations of Proposition \ref{propavecmyersonfaible}} \label{illusprop}

First, let us illustrate Proposition \ref{propavecmyersonfaible} (b) in the case where $d$ is a prime number, with the new insight brought by Proposition \ref{interpret}. \\

Let $d$ be a prime number. For all $d$-admissible integer $q$, we consider the sums

$$\Sim_q(a,d):=\sum_{\substack{x \in (\zqz)^\times \\x^d = 1}}^{} e\left(\frac{ax}{q}\right) \text{ for }  a \in \zqz.$$
Since the vector $\mm =(1) \in \Z^1$ is coprime with $d$, Proposition \ref{propavecmyersonfaible} (b) states that these sums all belong to the image of $\T^{d-1}$ via $g_d$, and that the sets:
$$\mathcal S_q(-,d) := \left\{ \Sim_q(a,d);  \ a \in \zqz \right\}$$
become equidistributed in this image, with respect to the pushforward measure of the Haar measure on $\T^{d-1}$. Now, the new insight given by Proposition \ref{interpret} is the interpretation of the image of $g_d$ as the region $\H_d$ delimited by a $d$-cusp hypocycloid.\\
 We illustrate this statement in the case $d=3$. In the figure below, the blue points are the sets $\mathcal S_q(-,3)$ while the red curve is the $3$-cusp hypocycloid from Definition \ref{defhypo}.
\begin{figure}[H]
	\centering
	\begin{subfigure}[b]{0.32\textwidth}
		\centering
		\includegraphics[width=\textwidth]{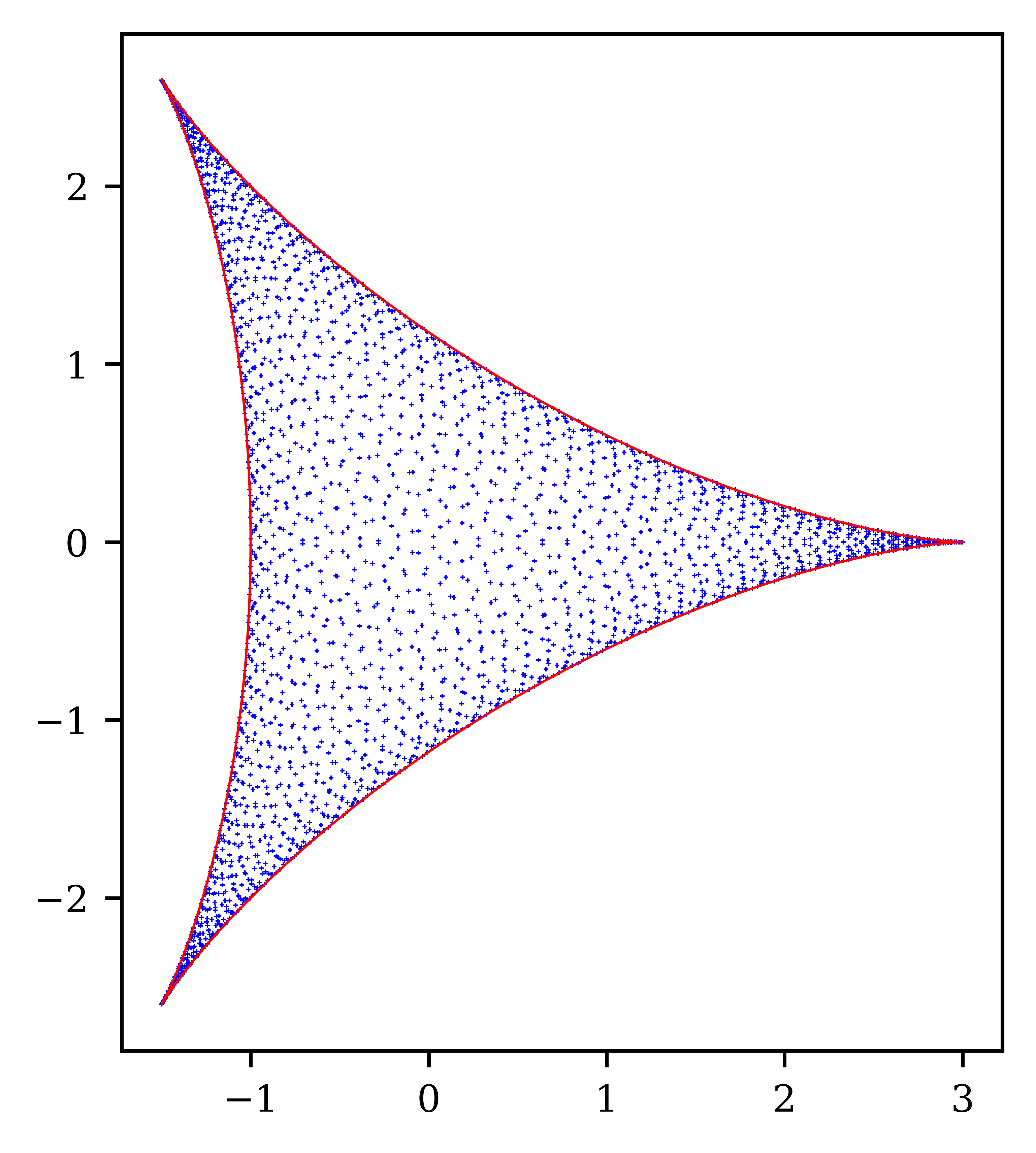}
		\caption{$q = 7759$}
	\end{subfigure}
	\hfill
	\begin{subfigure}[b]{0.32\textwidth}
		\centering
		\includegraphics[width=\textwidth]{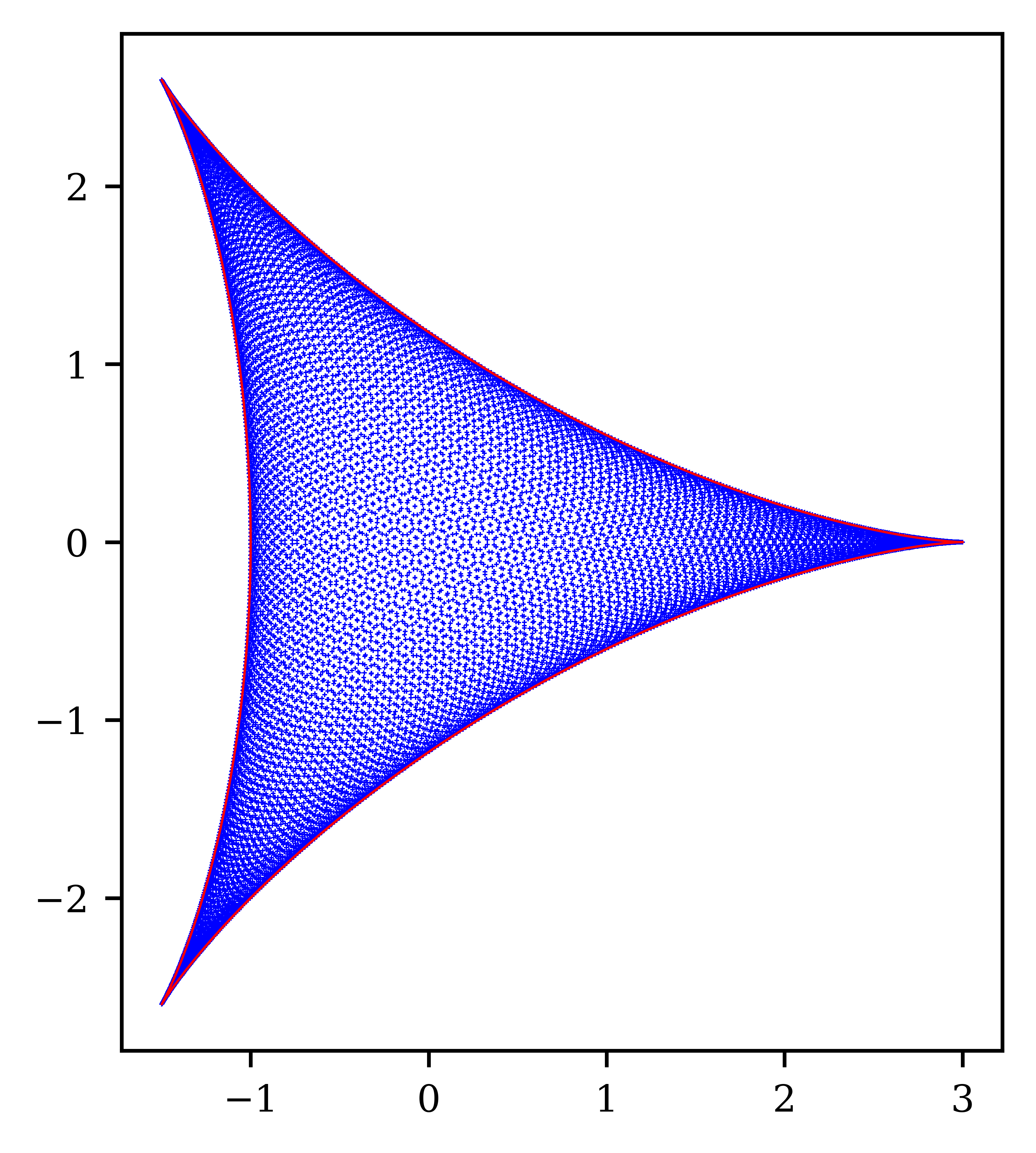}
		\caption{$q= 51361$}
	\end{subfigure}
	\hfill
	\begin{subfigure}[b]{0.32\textwidth}
		\centering
		\includegraphics[width=\textwidth]{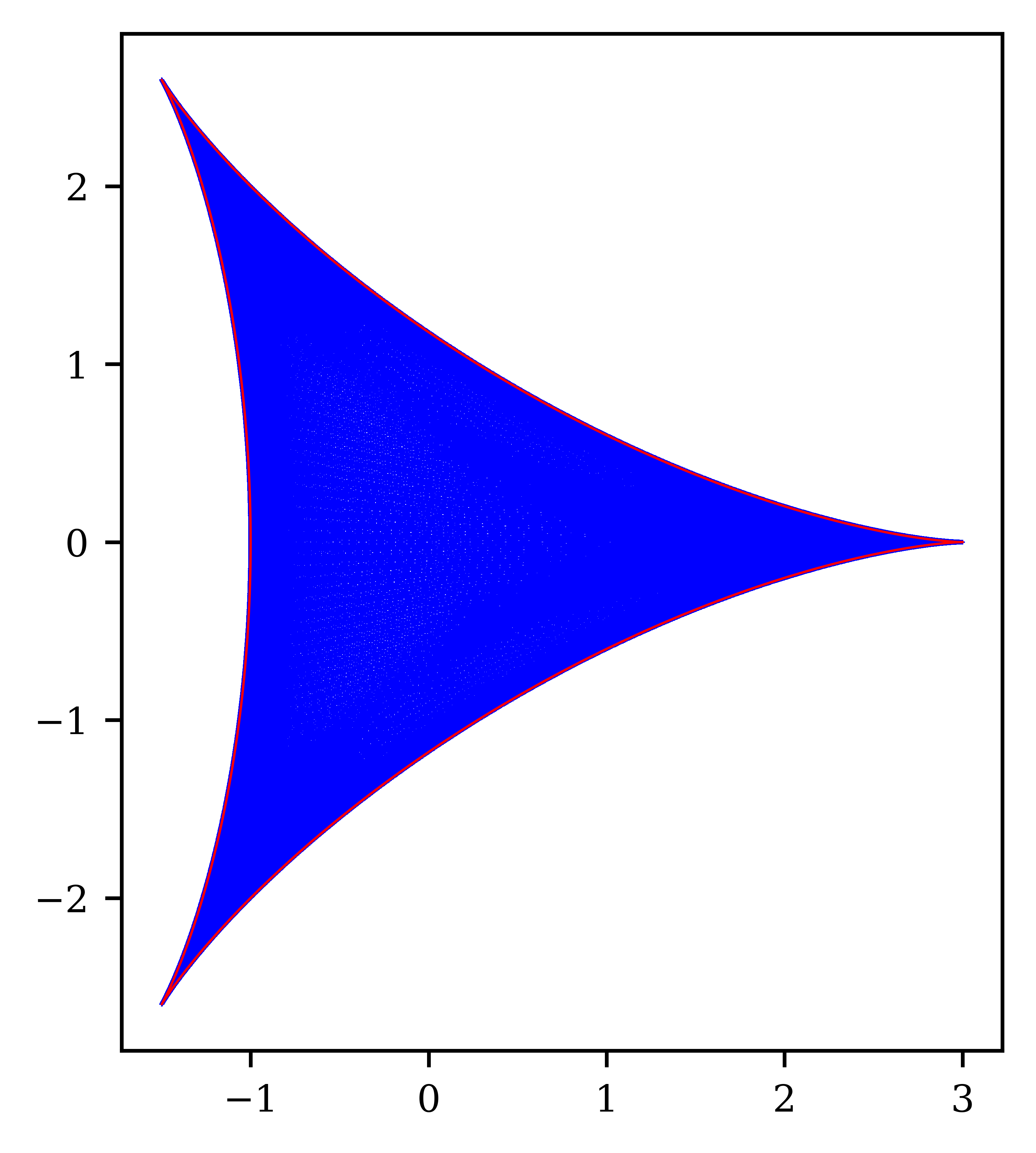}
		\caption{$q=326041 =571^2$}
	\end{subfigure}
	\caption{The sets $\mathcal S_q(-,d)$ for $d= 3$ and three $3$-admissible values of $q$.}
	\label{simpled3}
\end{figure}

However, Proposition \ref{propavecmyersonfaible} (b) covers many other families of exponential sums. For instance, if one takes $\mm$ to be the vector $(1,-1) \in \Z^2$, the same equidistribution phenomenon happens. Indeed, $\mm =(-1,1)$ is coprime with $d$ (for any $d$) and so the proposition also predicts that the sets of sums

$$\mathcal K_q(-,-,d) := \left\{ \K_q(a,b,d) := \sum_{\substack{x \in (\zqz)^\times \\ x^d = 1}} e\left(\frac{ax + bx^{-1}}{q}\right), \quad  a,b \in \zqz \right\} $$

become equidistributed in $\H_d$ as $q$ goes to infinity among the $d$-admissible integer (and with respect to the same measure as in the previous example). Figure \ref{3hypo5branches} of the introduction illustrates this result in the case $d= 5$. In the case where $d= 3$, the comparison of the figure below and Figure \ref{simpled3} illustrates this striking similitude of behaviour for different types of exponential sums, when restricted to subgroups.

\begin{figure}[H]
	\centering
	\begin{subfigure}[b]{0.32\textwidth}
		\centering
		\includegraphics[width=\textwidth]{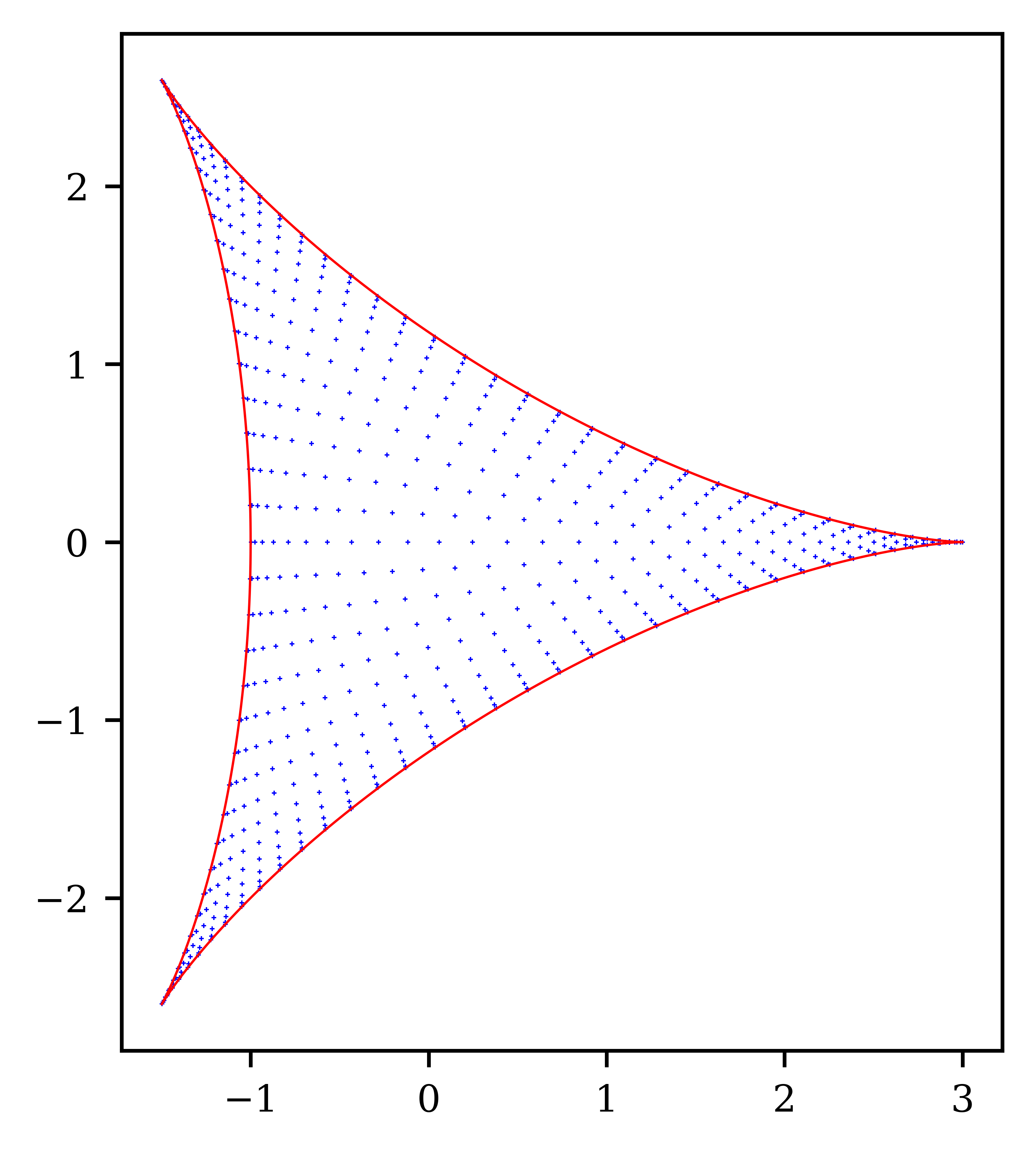}
		\caption{$q = 61$}
	\end{subfigure}
	\hfill
	\begin{subfigure}[b]{0.32\textwidth}
		\centering
		\includegraphics[width=\textwidth]{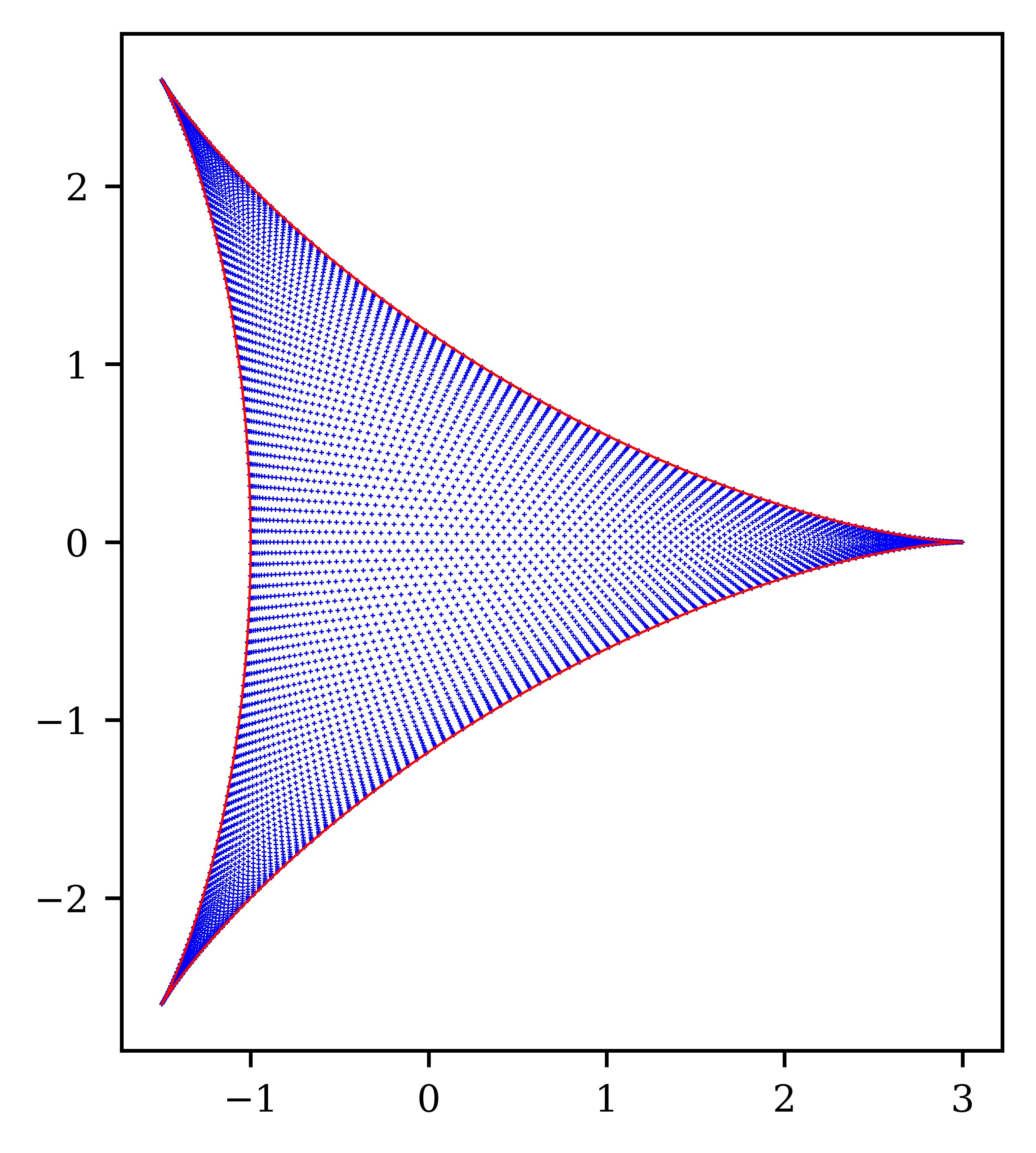}
		\caption{$q= 199$}
	\end{subfigure}
	\hfill
	\begin{subfigure}[b]{0.32\textwidth}
		\centering
		\includegraphics[width=\textwidth]{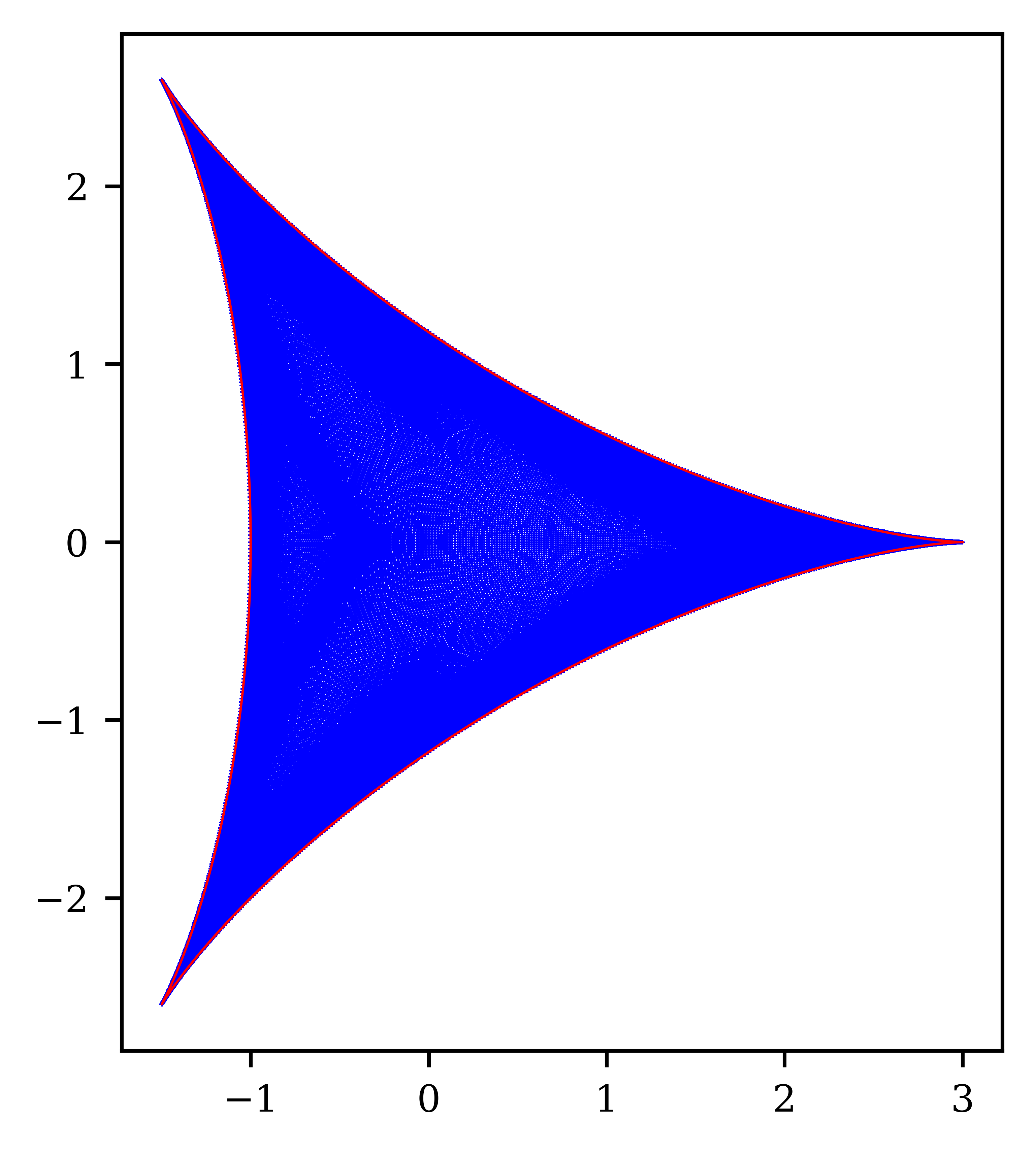}
		\caption{$q=643$}
	\end{subfigure}
	\caption{The sets $\mathcal K_q(-,-,d)$ for $d= 3$ and three $3$-admissible values of $q$.}
	\label{kloosd3}
\end{figure}

Proposition \ref{propavecmyersonfaible} (b) also states that one can fix $a$ and let $b$ vary in all $\zqz$, and vice-versa. For instance, if we take the previous example of restricted Kloosterman sums and choose to fix $a=1$ and let only $b$ vary, the sets of sums:

$$\mathcal K_q(1,-,d) := \{ \K_q(1,b,d); \  b\in \zqz \}$$
will become equidistributed in the same hypocycloid as before, with respect to the same measure.

%

Thus, Proposition \ref{propavecmyersonfaible} allows us to recover the equidistribution results from \cite{periods, menagerie, visual} and extends \cite[Theorem 7 and Theorem 10]{visual} to all values of $d$, and to the case where only one of the two parameter $a$ and $b$ varies.\\

Moreover, Proposition \ref{propavecmyersonfaible} (b) widely generalizes the previously known results to other families of exponential sums. Indeed, sums with $ax$ or $ax+bx^{-1}$ inside the exponentials may now be replaced by sums with $a_1 x^{m_1} + \dots a_n x^{m_n}$ inside the exponentials, provided the $m_i$ are coprime with $d$. For instance, one can consider the sums

$$\QQ_q(a,b,c,d) := \sum_{\substack{x \in (\zqz)^\times \\ x^d = 1}} e\left(\frac{ax^4 + bx^2 +cx}{q}\right) \text{ for } a,b,c \in \zqz, $$
for all $d$-admissible integer $q$. As soon as $d$ is odd, it is coprime with the exponents of $X$ that appear in the polynomials of the form $$aX^4 + bX^2 + cX.$$ Therefore, Proposition \ref{propavecmyersonfaible} (b) applies to this family of sums, as long as the summation is restricted to a subgroup of odd order. So if we look at the case $d=3$ and we draw the sets
$$\mathcal Q_q(-,-,-,3) = \{ \QQ_q(a,b,c,3); \ a,b,c \in \zqz \} $$
for different $3$-admissible values of $q$, we observe the same equidistribution as for the other types of sums, inside a $3$-cusp hypocycloid.
\begin{figure}[H]
	\centering
	\begin{subfigure}[b]{0.3\textwidth}
		\centering
		\includegraphics[width=\textwidth]{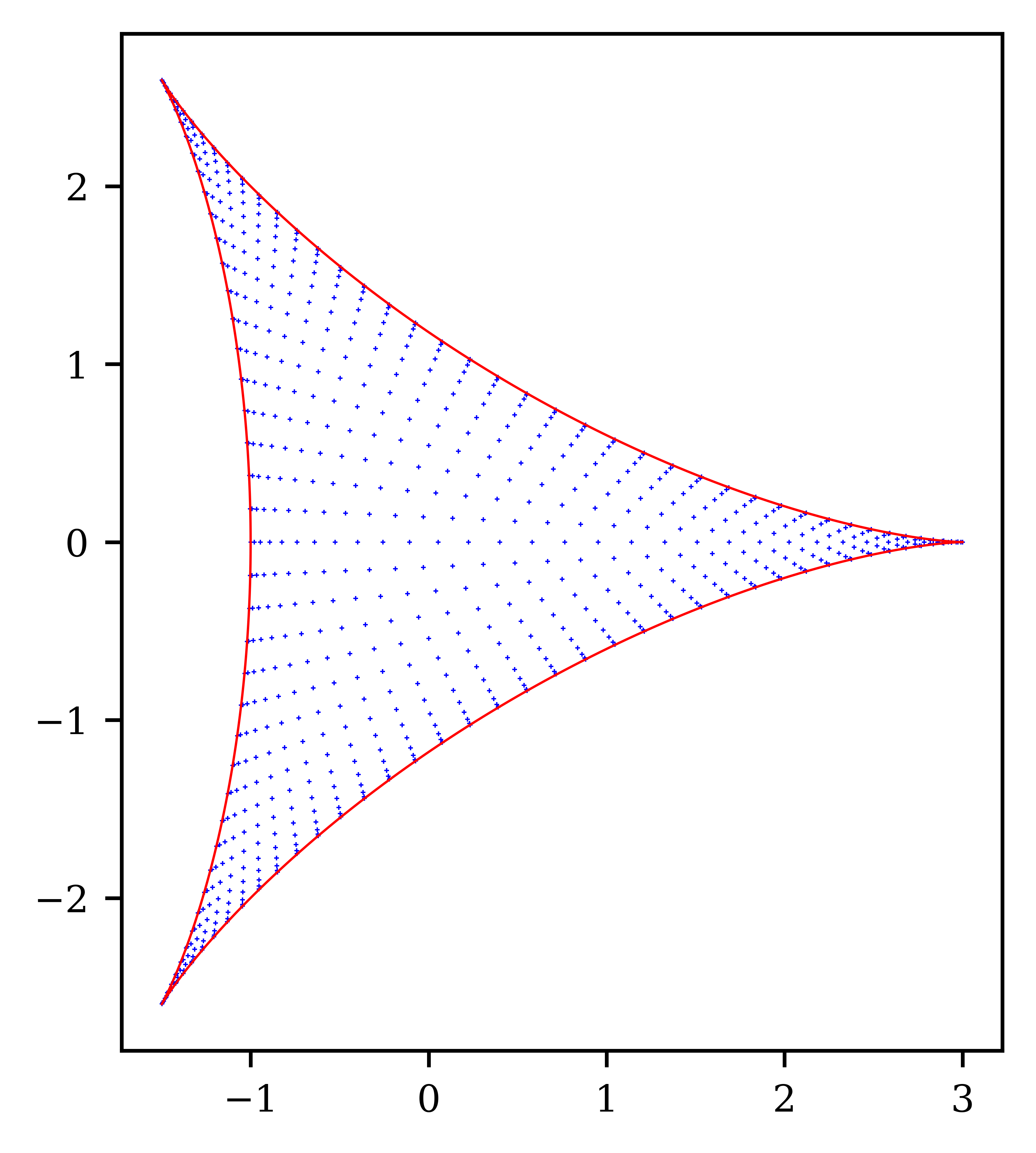}
		\caption{$q = 67$}
	\end{subfigure}
	\hfill
	\begin{subfigure}[b]{0.3\textwidth}
		\centering
		\includegraphics[width=\textwidth]{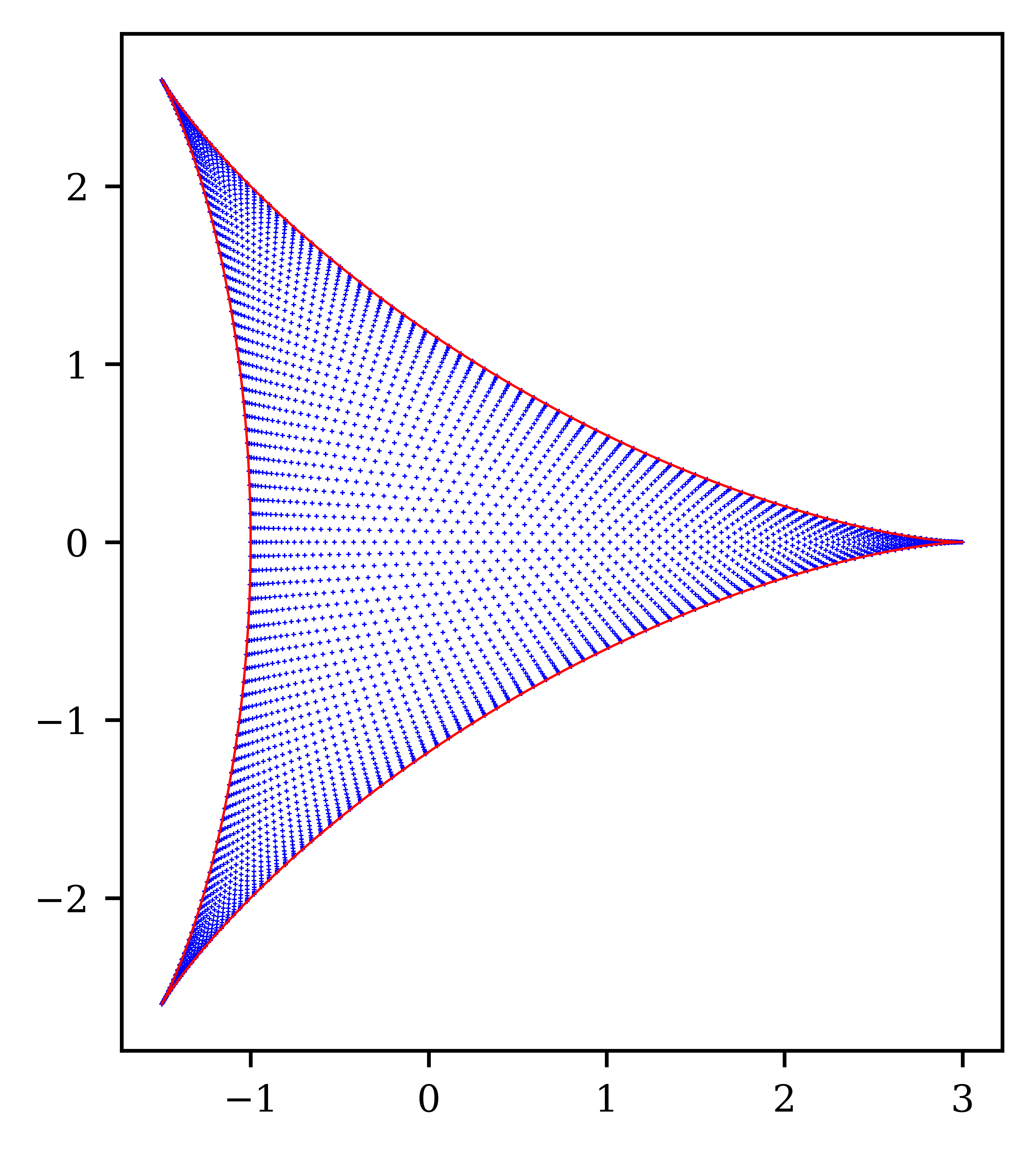}
		\caption{$q= 157$}
	\end{subfigure}
	\hfill
	\begin{subfigure}[b]{0.3\textwidth}
		\centering
		\includegraphics[width=\textwidth]{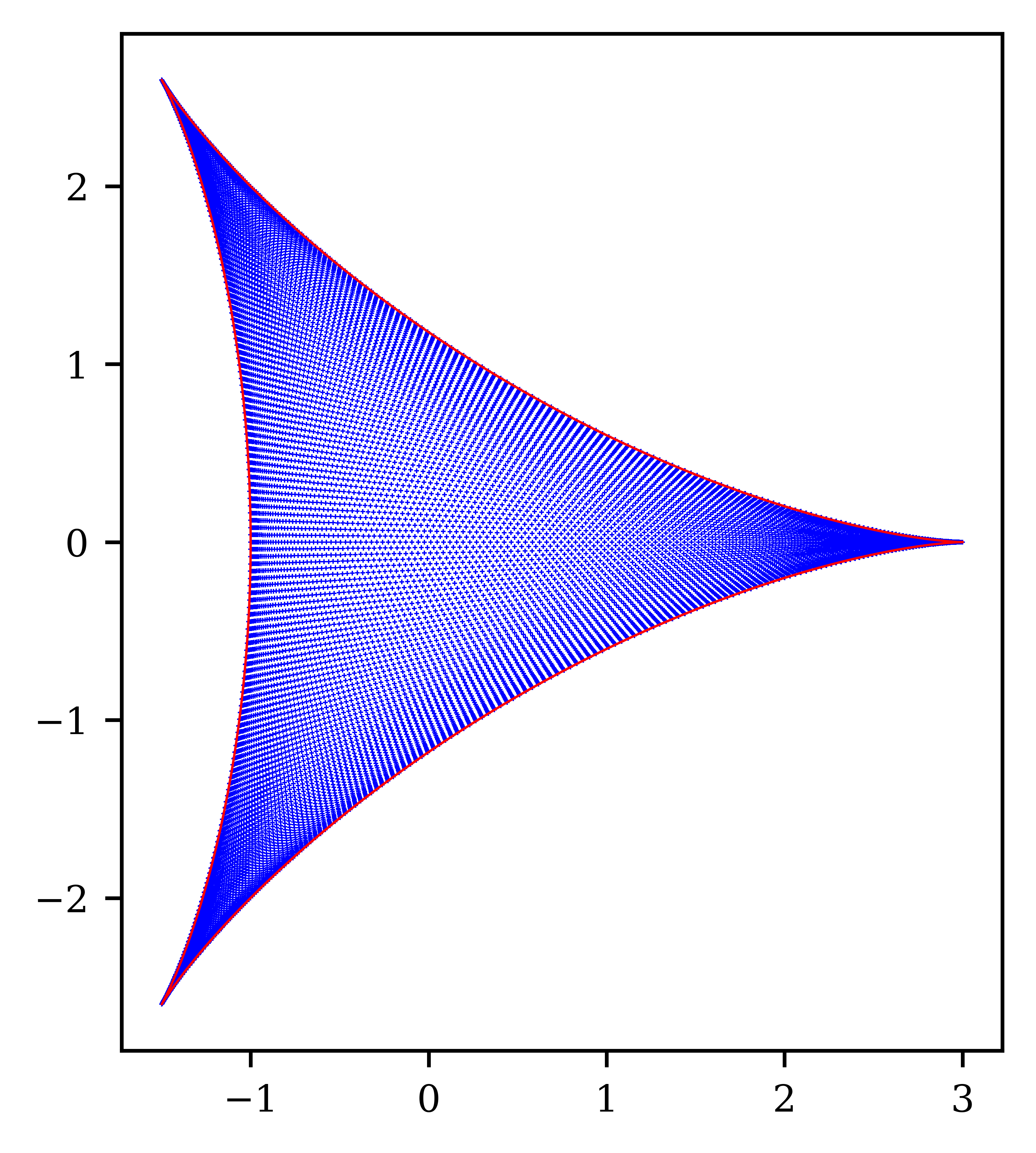}
		\caption{$q=307$}
	\end{subfigure}
	\caption{The sets $\mathcal Q_q(-,-,-,d)$ for $d= 3$ and three $3$-admissible values of $q$.}
	\label{abcd3}
\end{figure}

One could also want to consider sets of Birch sums restricted to a subgroup, that is:
$$\B_q(a,b,d) := \sum_{\substack{x \in (\zqz)^\times\\ x^d =1} }^{} e\left(\frac{ax^3 + bx}{q}\right) \text{ où } a,b \in \zqz$$
For instance if we take $d=7$ and look at the sets $ \mathcal B_q(-,-,7) :=\left\{ \B_q(a,b,7) ; \ a,b \in \zqz \right\}$, then Proposition \ref{propavecmyersonfaible} (b) (combined with Proposition \ref{interpret}) states that they should become equidistributed in $\H_7$ (the region delimited by a $7$-cusp hypocycloid) as $q$ goes to infinity among the $7$-admissible integers. This is indeed what the following pictures suggest:
\begin{figure}[H]
	\centering
	\begin{subfigure}[b]{0.3\textwidth}
		\centering
		\includegraphics[width=\textwidth]{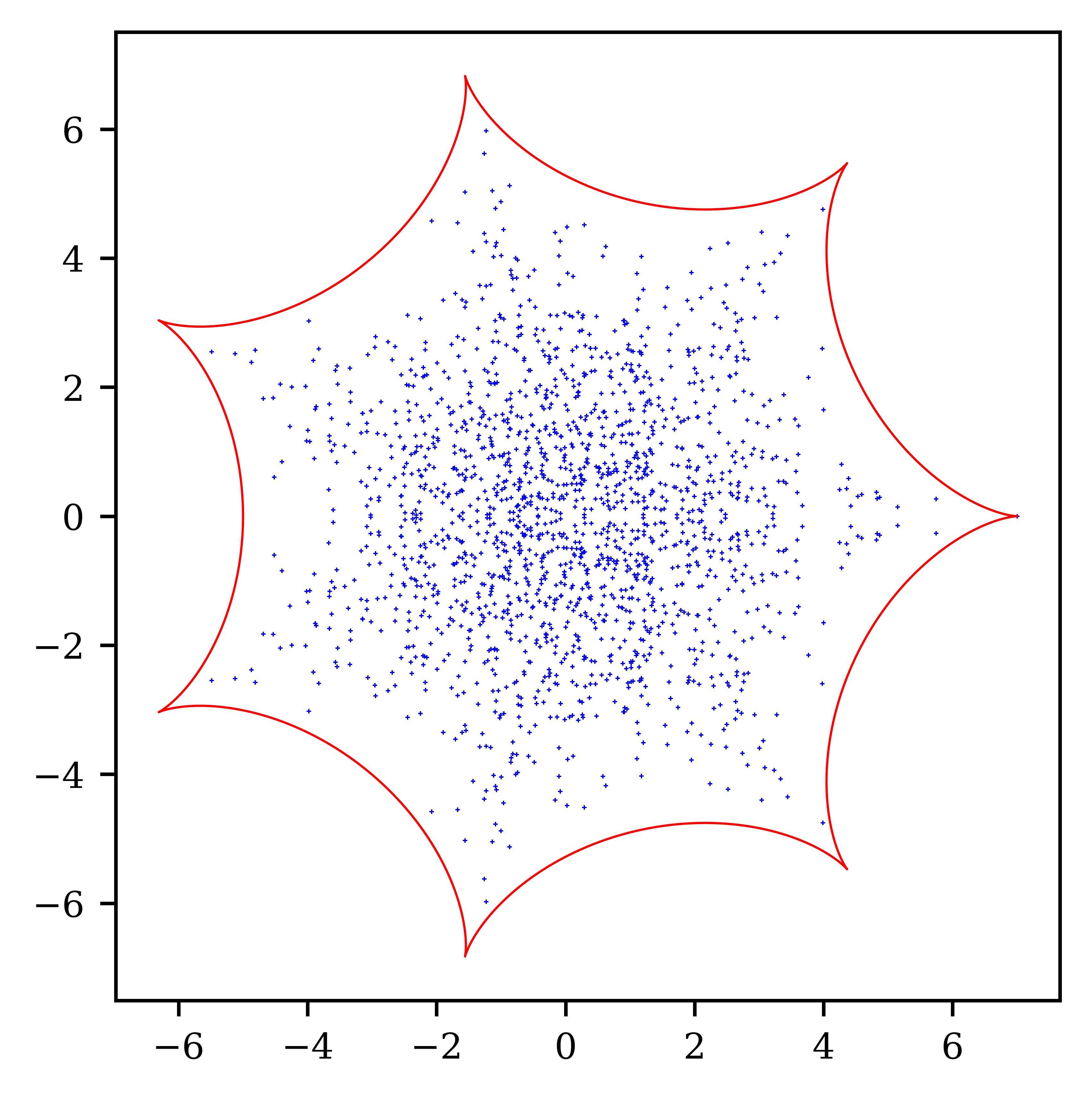}
		\caption{$q = 113$}
	\end{subfigure}
	\hfill
	\begin{subfigure}[b]{0.3\textwidth}
		\centering
		\includegraphics[width=\textwidth]{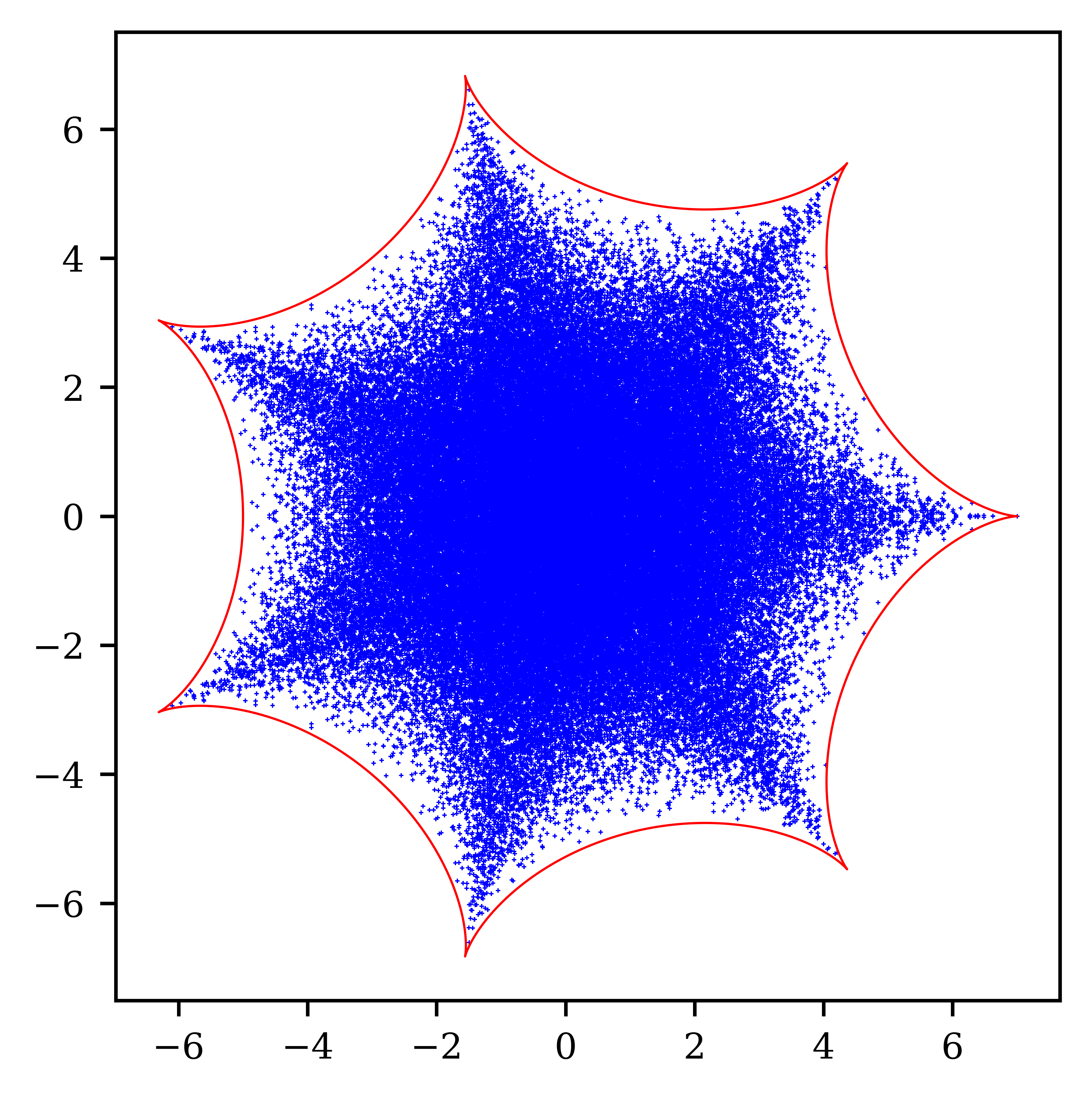}
		\caption{$q= 827$}
	\end{subfigure}
	\hfill
	\begin{subfigure}[b]{0.3\textwidth}
		\centering
		\includegraphics[width=\textwidth]{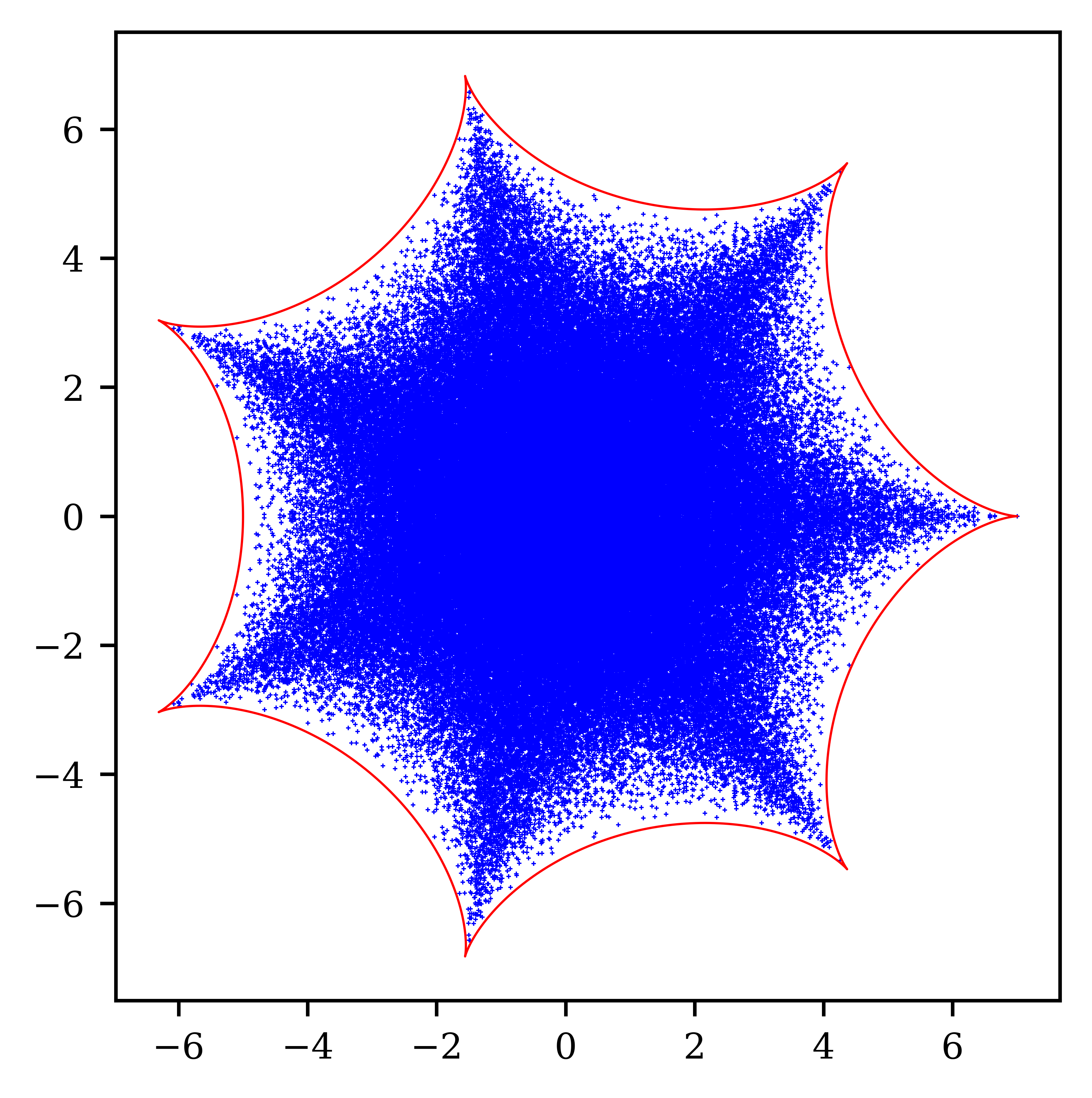}
		\caption{$q=1009$}
	\end{subfigure}
	\caption{The sets $\mathcal B_q(-,-,d)$ for $d= 7$ and three $7$-admissible values of $q$.}
	\label{birchd7}
\end{figure}

\begin{rem}
Note that in Theorem \ref{thprincipal} and Proposition \ref{propavecmyersonfaible}, the measure with respect to which the sums become equidistributed is the pushforward measure via $g_d$ of the Haar measure on $\T^{\varphi(d)}$. This explains why one does not observe a \enquote{uniform} distribution in the sense of the Lebesgue measure.
\end{rem}

On the other hand, one could want to consider Birch sums restricted to the subgroup of order $3$, that is sums of the type:

$$\B_q(a,b,3) := \sum_{\substack{x \in (\zqz)^\times\\ x^3 =1} }^{} e\left(\frac{ax^3 + bx}{q}\right) \text{ où } a,b \in \zqz$$

However, this type of sum does not fall inside the range of application of Proposition \ref{propavecmyersonfaible} (b), because the exponent $3$ in the polynomial expression $ax^3+bx$ is not coprime with the order of the subgroup. In fact, this situation is part of case (a) of Proposition \ref{propavecmyersonfaible}.\\
 
Finally, let us illustrate Proposition \ref{propgrb}, which gives a geometric interpretation of the image of $g_d$ when $d$ is a prime power.\\
	
	If we take again the example of Kloosterman sums, Proposition \ref{propavecmyersonfaible} (b) states that the sums:
	$$ \K_q(a,b,9) = \sum_{\substack{x \in (\zqz)^\times \\ x^9 = 1}} e\left(\frac{ax + bx^{-1}}{q}\right), \quad  a,b \in \zqz $$
	become equidistributed in the image of $g_9$, with respect to the pushforward measure of the Haar measure on $\T^{6}$. Now, thanks to Proposition \ref{propgrb} we can interpret the image of $g_9$ as the Minkowski sum of three copies of $\H_3$ (see also example \ref{g9}). Therefore, we should observe equidistribution in a region of the shape given by \cite[Figure 11]{menagerie}. The figure below illustrates this asymptotic behaviour.
	
	\begin{figure}[H]
		\centering
		\begin{subfigure}[b]{0.32\textwidth}
			\centering
			\includegraphics[width=\textwidth]{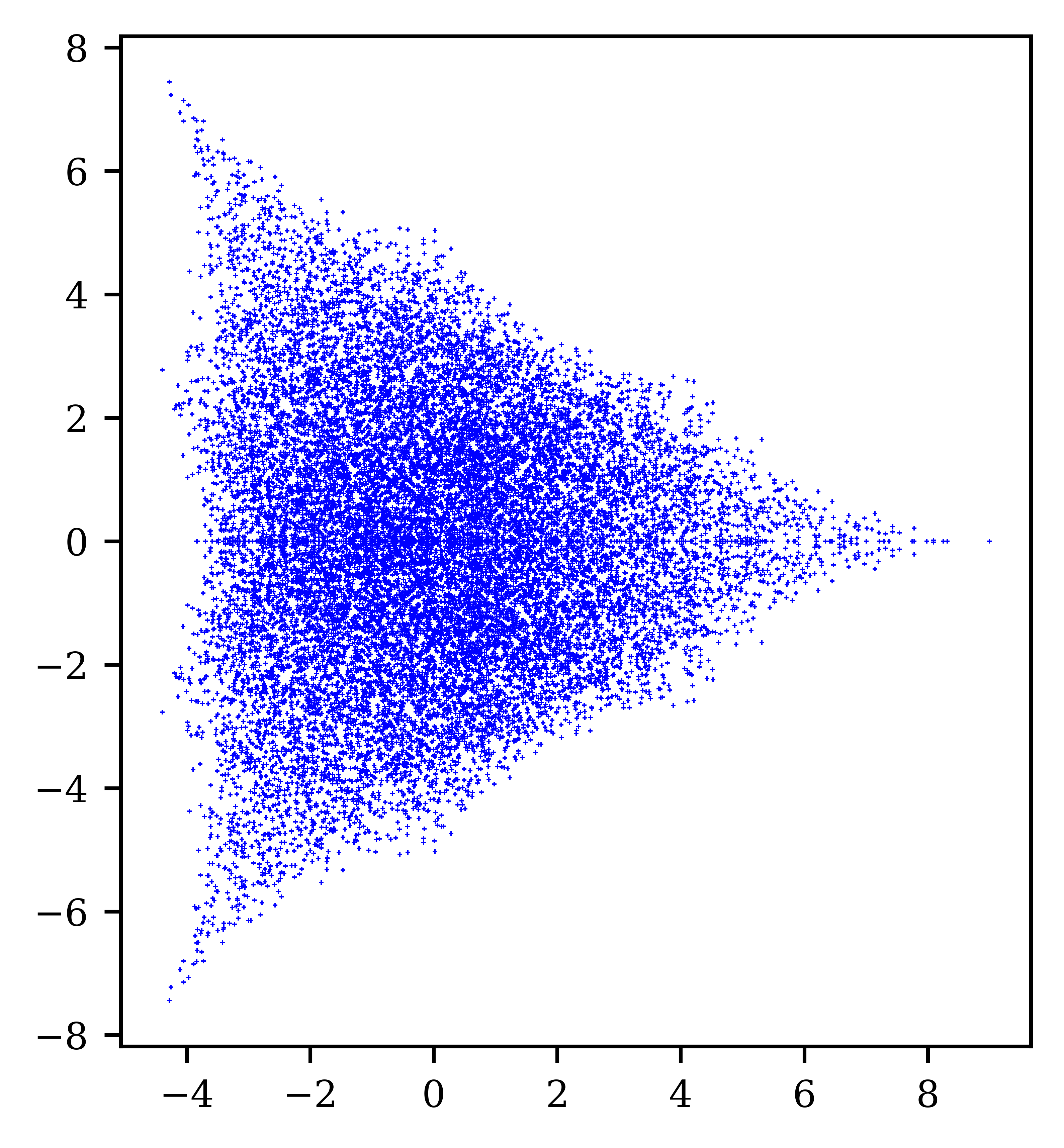}
			\caption{$q = 577$}
		\end{subfigure}
		\hfill
		\begin{subfigure}[b]{0.32\textwidth}
			\centering
			\includegraphics[width=\textwidth]{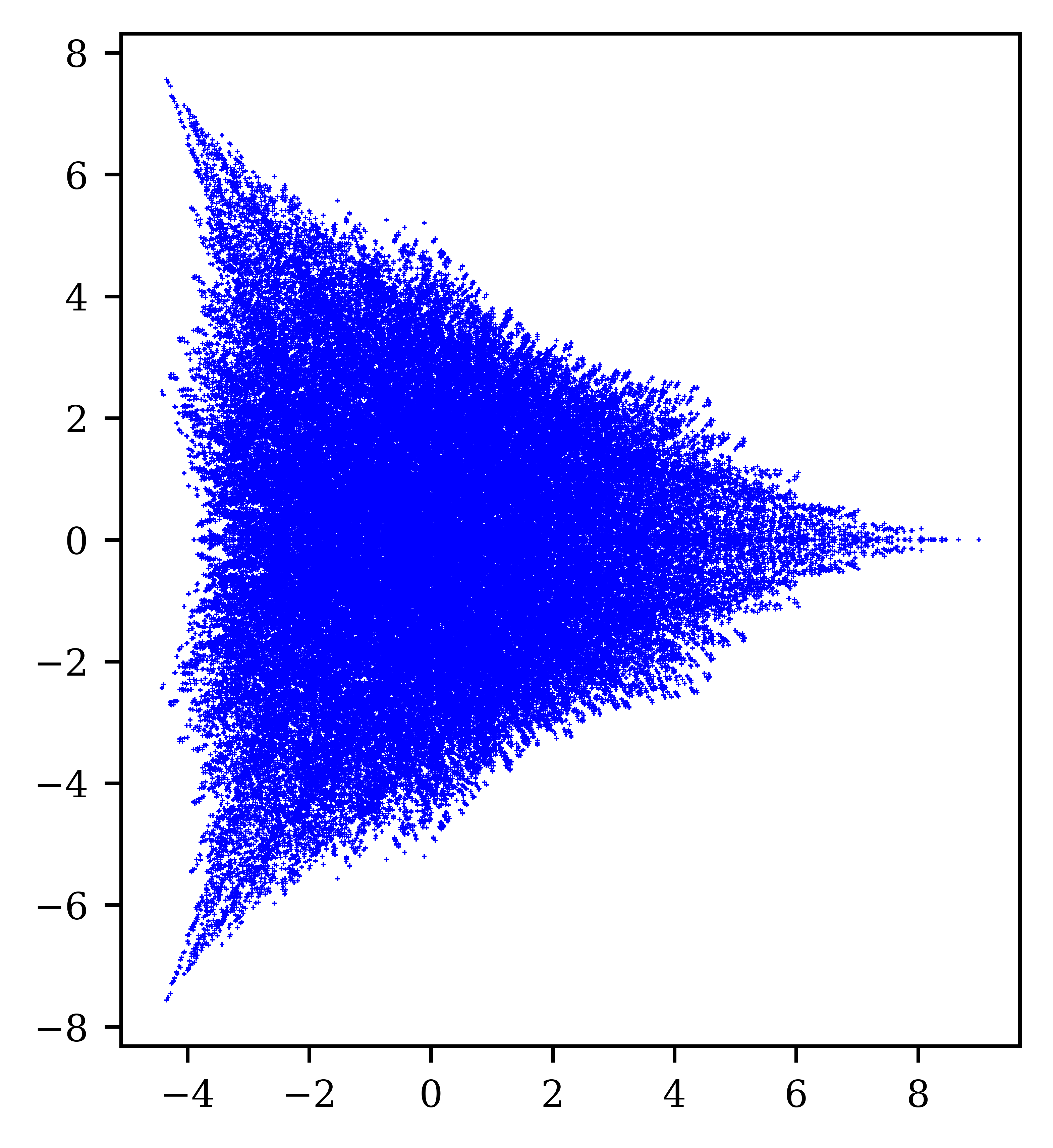}
			\caption{$q=1297$}
		\end{subfigure}
		\hfill
		\begin{subfigure}[b]{0.32\textwidth}
			\centering
			\includegraphics[width=\textwidth]{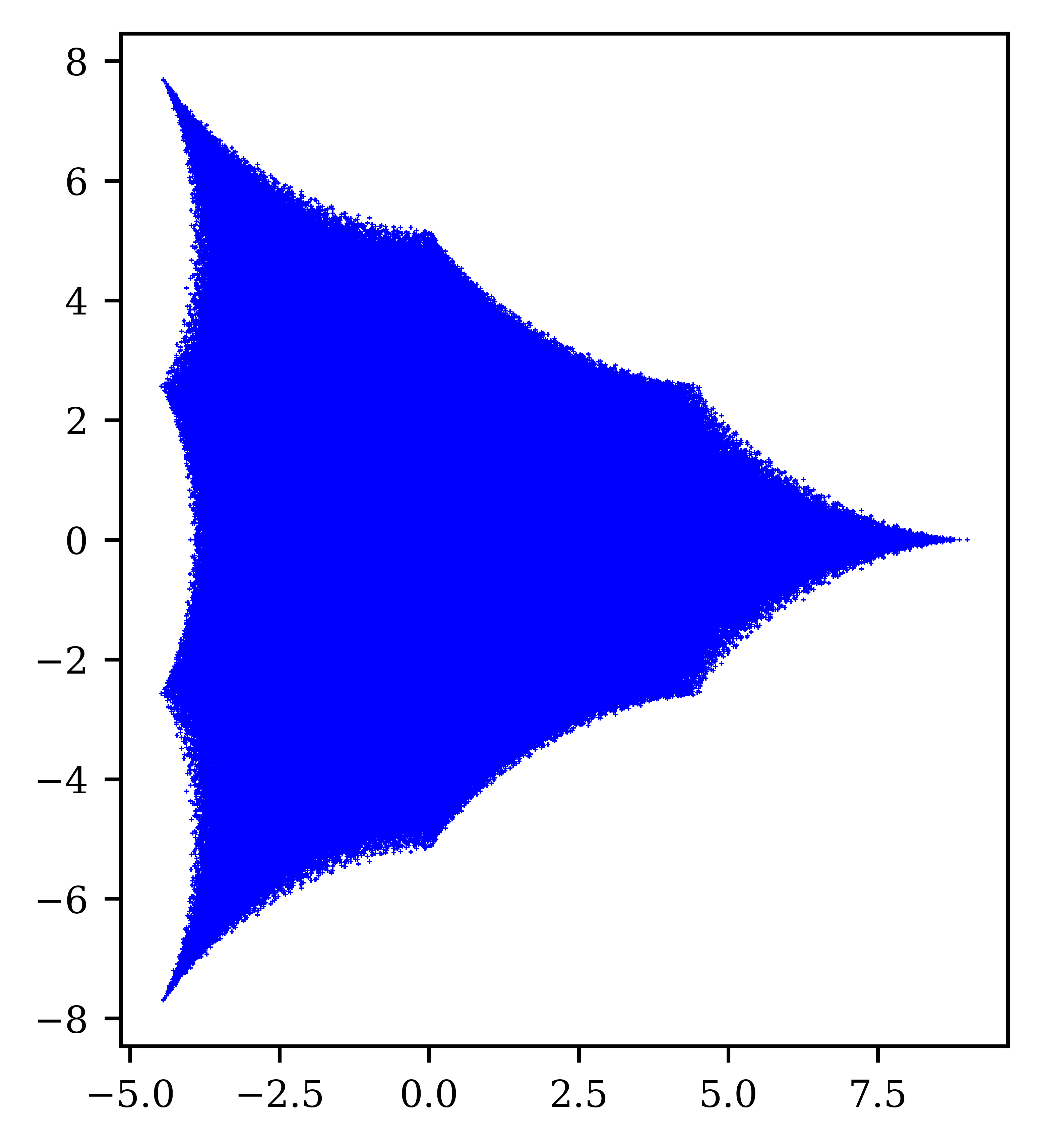}
			\caption{$q=73^2$}
		\end{subfigure}
		\caption{The sets $\mathcal K_q(-,-,9) :=  \left\{ \K_q(a,b,9); \  a,b \in \zqz \right\}$ for three $9$-admissible values of $q$.}
		
	\end{figure}

%


\bibliographystyle{alpha}
\bibliography{biblioarticle}

\vspace{1cm}
\hfill \begin{minipage}{0.9\textwidth}
\textsc{Université de Bordeaux, CNRS, Bordeaux INP, IMB, UMR 5251, F-33400 \\ Talence, France}.\\
\textit{Email address}: \texttt{theo.untrau@math.u-bordeaux.fr}
\end{minipage}

		
		

\end{document}